\documentclass[12pt]{article}
\usepackage{a4wide}
\usepackage{epsfig}
\usepackage{enumerate}
\usepackage{enumitem}
\usepackage{xcolor}
\usepackage{multirow}

\usepackage{amssymb}
\usepackage{amsmath}
\usepackage{amsfonts}
\usepackage{amsthm}
\newtheorem{theorem}{Theorem}
\newtheorem{lemma}[theorem]{Lemma}

\newtheorem{corollary}[theorem]{Corollary}
\newtheorem{conjecture}[theorem]{Conjecture}

\makeatletter
\newtheorem*{rep@theorem}{\rep@title}
\newcommand{\newreptheorem}[2]{%
\newenvironment{rep#1}[1]{%
 \def\rep@title{#2 \ref{##1}}%
 \begin{rep@theorem}}%
 {\end{rep@theorem}}}
\makeatother

\newreptheorem{theorem}{Theorem}
\newreptheorem{lemma}{Lemma}

\newtheoremstyle{problem-def}%
{}
{}
{}
{}
{\sc}
{.}
{.5em}
{\thmnote{#3}}
\theoremstyle{problem-def}
\newtheorem*{problem-def}{Problem}
\theoremstyle{plain}

\def\df#1{{\em #1\/}}

\def\C{\mathcal{C}}
\def\D{\mathcal{D}}

\def\G{\mathcal{G}}
\def\H{\mathcal{H}}
\def\L{\mathcal{L}}
\def\M{\mathcal{M}}

\def\P{\mathcal{P}}
\def\Q{\mathcal{Q}}

\def\T{\mathcal{T}}

\def\NN{\mathbb{N}}
\def\RR{\mathbb{R}}
\renewcommand{\SS}{\mathbb{S}}

\newcounter{cases}
\newcounter{subcases}[cases]

\newenvironment{casesblock}{%
\setcounter{cases}{0}\setcounter{subcases}{0}\par}{}

\long\def\case#1{\stepcounter{cases}\medskip\noindent{\bf Case \arabic{cases}: }#1. \par}

\def\ifempty#1#2#3{%
\def\ifemptytest{#1}%
\def\ifemptyempty{}%
\ifx\ifemptyempty\ifemptytest #2\else #3\fi%
}

\def\iso{\cong}

\renewcommand{\ss}{\subseteq}

\def\sm{\setminus}

\def\into{\to}

\newcommand{\name}[1]{{\sc #1}}

\def\Forb{\text{Forb}}
\def\g{g}
\def\gp{g^+}
\def\galt{g_a}
\def\gap{g_a^+}
\def\ep{\epsilon^+}
\def\e{\epsilon}
\def\hat#1{\widehat{#1}}
\def\dc#1#2{\Delta_{#1}(#2)}
\def\Gcxy{\G^\circ_{xy}}
\def\Cc{\C^\circ}
\def\h{h}

\begin{document}
\title{Obstructions of Connectivity 2 for Embedding Graphs into the Torus}
\author{%
  Bojan Mohar
  \thanks{Supported in part by an NSERC Discovery Grant (Canada),
    by the Canada Research Chair program, and by the
    Research Grant P1--0297 of ARRS (Slovenia).}~\thanks{On leave from:
    IMFM \& FMF, Department of Mathematics, University of Ljubljana, Ljubljana,
    Slovenia.}
  \qquad
  Petr \v{S}koda\\\\
  \normalsize
  Department of Mathematics, \\
  \normalsize
  Simon Fraser University,\\
  \normalsize
  8888 University Drive,\\ 
  \normalsize
  Burnaby, BC, Canada.}
\date{}
\maketitle

\begin{abstract}
The complete set of minimal obstructions for embedding graphs into the torus is still not determined.
In this paper, we present all obstructions for the torus of connectivity 2. Furthermore, we
describe the building blocks of obstructions of connectivity 2 for any orientable surface.
\end{abstract}

\section{Introduction}

\def\ss{\subseteq}
\def\L{\mathcal{L}}
\def\SS{\mathbb{S}}

The problem which graphs can be embedded in a given surface is a fundamental question
in topological graph theory. Robertson and Seymour~\cite{robertson-1990} proved that for each surface $\SS$
the class of graphs that embeds into $\SS$ can be characterized by a finite list $\Forb(\SS)$ of minimal forbidden minors (or \df{obstructions}).
For the 2-sphere $\SS_0$, $\Forb(\SS_0)$ consists of the \df{Kuratowski graphs}, $K_5$ and $K_{3,3}$. The list of obstructions $\Forb(\NN_1)$
for the projective plane $\NN_1$ already contains 35 graphs and $\NN_1$ is the only other surface for which the complete list of forbidden minors is known.
The number of obstructions for both orientable and non-orientable surfaces seems to grow fast with the genus
and that can be one of the reasons why even  for the torus $\SS_1$ the complete list of obstructions is still not known, although thousands of obstructions
were generated by the computer (see~\cite{gagarin-2009}).

In this paper, we study the obstructions for orientable surfaces of low connectivity.
It is easy to show that obstructions that are not 2-connected can be obtained as disjoint unions and 1-sums of obstructions for surfaces of smaller genus (see~\cite{battle-1962}).
Stahl~\cite{stahl-1977} and Decker et al.~\cite{decker-1981} showed that genus of 2-sums differs by at most 1 from the sum of genera of its parts.
Decker et al.~\cite{decker-1985} provided a simple formula for the genus of a 2-sum that will be used in this paper.
We shall prove that obstructions for an orientable surface of connectivity 2 can be obtained as a 2-sum of building blocks that fall (roughly) into two families of graphs.
One family consists of obstructions for embeddings into surfaces of smaller genus.
The graphs in the second family are critical with respect to the graph parameter $\galt$ defined in Sect.~\ref{sc-terminals}.
We use this characterization in Sect.~\ref{sc-torus} to construct all obstructions for the torus of connectivity 2.

\section{Notation}

Let $G$ be a connected multigraph. An (orientable) \df{embedding} $\Pi$ of $G$ is a mapping that assigns to each vertex $v \in V(G)$ 
a cyclic permutation $\pi_v$ of the edges incident with $v$, called the \df{local rotation} around $v$. 
If $\Pi$ is an embedding of $G$, then we also say that $G$ is \df{$\Pi$-embedded}.
Given a $\Pi$-embedded graph $G$, a \df{$\Pi$-face} (or \df{$\Pi$-facial walk}) is a cyclic sequence $(v_1, e_1, \ldots, v_k, e_k)$ such that $e_i = v_iv_{i+1}$,
$e_i = \pi_{v_i}(e_{i-1})$ for each $i = 1,\ldots, k$ (where $v_{k+1} = v_1$ and $e_0 = e_k$), and all pairs $(v_i, e_i)$ are distinct.
The linear subsequence $e_{i-1}, v_i, e_i$ in a $\Pi$-face is called a \df{$\Pi$-angle} at $v_i$.
Note that edges of $\Pi$-angles are formed precisely by the pairs of edges that are consecutive in the local rotation around a vertex.

Each edge $e$ of a $\Pi$-embedded graph appears twice in the $\Pi$-faces. If there exists a single $\Pi$-face where $e$ appears twice,
we say that $e$ is \df{singular}. Otherwise, $e$ is \df{non-singular}.

The \df{genus} of an orientable embedding $\Pi$ of a graph $G$ is given by the Euler formula,
\begin{equation}
  \label{eq-euler}
  g(\Pi) = \frac{1}{2}(2 - n + m - f),
\end{equation}
where $n$ is the number of vertices, $m$ the number of edges and $f$ the number of $\Pi$-faces of $G$.
The \df{genus} $g(G)$ of a connected multigraph $G$ is the minimum genus of an orientable embedding of $G$.

In this paper, we will deal mainly with the class $\G$ of simple graphs. 
Let $G \in \G$ be a simple graph and $e$ an edge of $G$. 
Then $G-e$ denotes the graph obtained from $G$ by \df{deleting} $e$
and $G/e$ denotes the graph\footnote{When contracting an edge, one may obtain multiple edges. 
We shall replace any multiple edges by single edges as such a simplification has no effect on the genus.} 
obtained from $G$ by \df{contracting} $e$. It is convenient for us to formalize these graph operations.
The set $\M(G) = E(G) \times \{-,/\}$ is the \df{set of minor-operations} available for $G$.
An element $\mu \in \M(G)$ is called a \df{minor-operation} and $\mu G $ denotes the graph obtained
from $G$ by applying $\mu$. For example, if $\mu = (e, -)$ then $\mu G  = G - e$.
A graph $H$ is a \df{minor} of $G$ if $H$ can be obtained from a subgraph of $G$ by contracting some edges.
If $G$ is connected, then $H$ can be obtained from $G$ by a sequence of minor-operations.

Let $H$ be a subgraph of $G$. We say that $H$ is \df{minor-tight} (for the genus parameter $g$) if 
$g(\mu G) < g(G)$ for every minor-operation $\mu \in \M(H)$. The following observation
asserts that being an obstruction for  a surface is equivalent to having all subgraphs minor-tight.

\begin{lemma}
\label{lm-minor-tight}
Let $H_1, \ldots, H_s$ be subgraphs of a graph $G$ with $g(G) = k+1$.
If $E(H_1) \cup \cdots \cup E(H_s) = E(G)$, then
$G$ is an obstruction for\/ $\SS_k$ if and only if $H_1, \ldots, H_s$ are minor-tight.
\end{lemma}

It is well-known that each closed orientable surface is homeomorphic, for some $k \ge 0$, to the surface $\SS_k$,
which is the surface obtained from the sphere by adding $k$ handles.
A graph with an embedding of genus $k$ can be viewed as embedded onto $\SS_k$ (see~\cite{mohar-book}).

A graph has \df{connectivity} $k$ when it is $k$-connected but not $(k+1)$-connected.
An edge whose deletion disconnects the graph is a \df{cut-edge}.
The structure of obstructions for orientable surfaces that have connectivity at most 1 is very simple.
They are disjoint unions and 1-sums of obstructions for surfaces of smaller genus. 
This can be easily seen as an application of the following theorem that states that the 
genus of graphs is additive with respect to their 2-connected components (or \df{blocks}).

\begin{theorem}[Battle et al.~\cite{battle-1962}]
\label{th-battle-ori}
The genus of a graph is the sum of the genera of its blocks.
\end{theorem}

\section{Graphs with terminals}
\label{sc-terminals}

In this paper, we study obstructions for embedding graphs into orientable surfaces 
that have connectivity 2. 
Given graphs $G_1$ and $G_2$ such that $V(G_1) \cap V(G_2) = \{x, y\}$ and $xy \not\in E(G_1), E(G_2)$,
we say that each of the graphs $G = (V(G_1) \cup V(G_2), E(G_1) \cup E(G_2))$ and $G + xy$ is an \df{$xy$-sum} of $G_1$ and $G_2$.
We shall always specify if the $xy$-sum contains the edge $xy$ or not. The graphs $G_1$ and $G_2$ are the \df{parts} of the $xy$-sum.
If $x$ and $y$ are not important, we sometimes refer to $G$ and $G+xy$ as \df{2-sums}.
\begin{figure}
  \centering
  \includegraphics{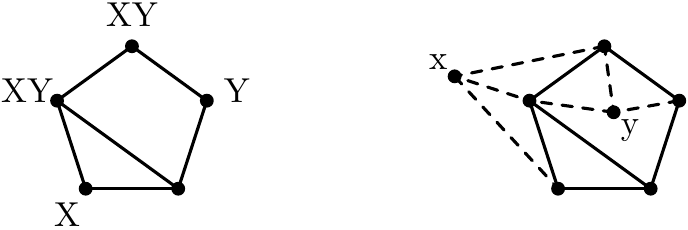}
  \caption{An example of an XY-labelled graph and its corresponding graph in $\Gcxy$.}
  \label{fg-xy-labelled}
\end{figure}

We wish to study the parts of a 2-sum separately and, in order to do so, we mark the vertices
of the separation as \df{terminals}. This prompts us to study the class of graphs $\G_{xy}$ with two terminals, $x$ and $y$.
The letters $x$ and $y$ will be consistently used for the two distinguished terminals.
Most notions that are used for graphs can be used in the same way for graphs with terminals.
Some notions differ though and, to distinguish between graphs with and without terminals, let $\hat{G}$ be the underlying graph of $G$ without terminals (for $G \in \G_{xy}$).
Two graphs, $G_1$ and $G_2$, in $\G_{xy}$ are \df{isomorphic} if there is an isomorphism of the graphs $\hat{G_1}$ and $\hat{G_2}$
that maps terminals of $G_1$ onto terminals of $G_2$ (and non-terminals onto non-terminals) possibly exchanging $x$ and $y$.
We define minor-operations on graphs in $\G_{xy}$ in the way that $\G_{xy}$ is a minor-closed family.
When performing edge contractions on $G \in \G_{xy}$, we do not allow contraction of the edge $xy$ (if $xy \in E(G)$)
and when contracting an edge incident with a terminal, the resulting vertex becomes a terminal. 
We use $\M(G)$ to denote the set of available minor-operations for $G$. 
Since $(xy, /) \not\in \M(G)$ for $G \in \G_{xy}$, 
we shall use $G /xy$ to denote the underlying simple graph in $\G$ obtained from $G$
by identification of $x$ and $y$; for this operation, we do not require the edge $xy$ to be present in~$G$.

For convenience, we use $\Gcxy$ for the subclass of $\G_{xy}$ of graphs without the edge $xy$.
We shall sometimes depict the graphs in $\Gcxy$ as XY-labelled graphs.
Given a graph $G \in \Gcxy$, let $H$ be the graph $G - x - y$ where a vertex of $H$ is labelled X if it is adjacent to $x$ in $G$
and it is labelled Y if it is adjacent to $y$ in $G$ (see Fig.~\ref{fg-xy-labelled}).
We say that $H$ is the \df{XY-labelled} graph \df{corresponding} to~$G$.

A \df{graph parameter} is a function $\G \into \RR$ that is constant on each isomorphism class of $\G$.
Similarly, we call a function $\G_{xy} \into \RR$ a \df{graph parameter} if it is constant on each isomorphism class of $\G_{xy}$.
A graph parameter $\P$ is \df{minor-monotone} if $\P(H) \le \P(G)$ for each graph $G \in \G_{xy}$ and each minor $H$ of $G$.
The graph genus is an example of a minor-monotone graph parameter.

Several other graph parameters will be used in this paper.
We use $G^+$ for the graph $G$ plus the edge $xy$ if it is not already present.
The genus of $G^+$ can be also viewed as a graph parameter $\gp$ defined as $\gp(G) = \g(G^+)$.
The graph parameter $\theta = \gp - g$ captures the difference between the genera of $G^+$ and $G$, that is $\theta(G) = \gp(G) - g(G)$.
Note that $\theta(G) \in \{0, 1\}$. 
\begin{figure}
  \centering
  \includegraphics{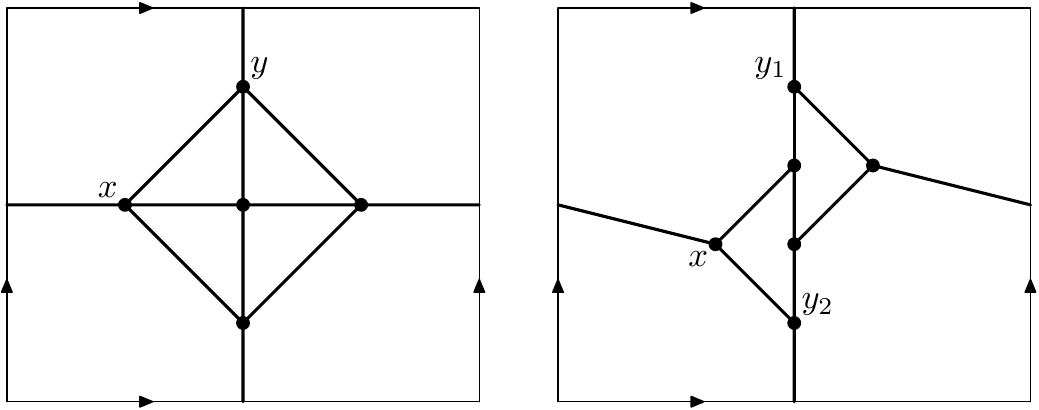}
  \caption{Kuratowski graphs and their two-vertex alternating embeddings in the torus.}
  \label{fg-kuratowski-alt}
\end{figure}

In order to compute the genus of an $xy$-sum of graphs, it is necessary to know whether $G$ has a minimum
genus embedding $\Pi$ with $x$ and $y$ appearing at least twice in an alternating order on a $\Pi$-face.
More precisely, we say that an embedding $\Pi$ is \df{$xy$-alternating} if there is a $\Pi$-face $W$ such that $(x,y,x,y)$ is a cyclic subsequence of $W$.
A graph $G \in \G_{xy}$ is \df{$xy$-alternating} if it has a minimum genus embedding that is $xy$-alternating. 
Fig.~\ref{fg-kuratowski-alt} shows two examples of $xy$-alternating embeddings in the torus.
We associate a graph parameter with this property. Let $\e(G) = 1$ if $G$ is $xy$-alternating and $\e(G) = 0$ otherwise.
We shall also use the graph parameter $\ep$ defined as $\ep(G) = \e(G^+)$.

In order to describe minimum genus embeddings of an $xy$-sum $G$ of graphs $G_1$ and $G_2$,
it is sufficient to consider two types of embeddings. To construct them, we take particular minimum genus embeddings $\Pi_1$ and $\Pi_2$ of $G_1$ and $G_2$ (respectively)
and combine them into an embedding $\Pi$ of $G$. 
For a non-terminal vertex $v$, let the local rotation around $v$ in $\Pi$ be the same as the local rotation around $v$ in $\Pi_i$ (if $v \in V(G_i)$ for $i \in \{1,2\}$).
Consider $\Pi_1$-faces $W_1$ and $W_2$ incident with $x$ and $y$, respectively,
and $\Pi_2$-faces $W_3$ and $W_4$ incident with $x$ and $y$, respectively.
Note that the faces $W_1$ and $W_2$ (and also $W_3$ and $W_4$) need not to be distinct.
We distinguish three cases.

\begin{casesblock}
  \case{$W_1, W_2, W_3, W_4$ are distinct faces}
  Write the face $W_1$ as $(x, e_1, U_1, e_2)$, $W_2$ as $(y, f_1, U_2, f_2)$, $W_3$ as $(x, e_3, U_3, e_4)$, and $W_4$ as $(y, f_3, U_4, f_4)$.
  Let $e_1, S_1, e_2$ be the linear sequence obtained from $\Pi_1(x)$ by cutting it at $e_1, e_2$.
  Similarly, let $e_3, S_2, e_4$ be the linear sequence obtained from $\Pi_2(x)$ by cutting it at an $e_3, e_4$.
  We let $\Pi(x)$ be the cyclic sequence $(e_1, S_1, e_2, e_3, S_2, e_4)$.
  Similarly, we define $\Pi(y)$ as the concatenation of the two linear sequences obtained from $\Pi_1(y)$ and $\Pi_2(y)$
  by cutting each of them at $f_1,f_2$ and $f_3, f_4$, respectively.
  Each $\Pi_1$-face and $\Pi_2$-face different from $W_1, W_2, W_3$, and $W_4$ is also a $\Pi$-face.
  The faces $W_1$ and $W_3$ combine into the $\Pi$-face $(x, e_1, U_1, e_2, x, e_3, U_3, e_4)$
  and the faces $W_2$ and $W_4$ combine into the $\Pi$-face $(y, f_1, U_2, f_2, y, f_3, U_4, f_4)$.
  Thus, the total number of faces decreased by two and~(\ref{eq-euler}) gives the following value of $g(\Pi)$:
  \begin{equation}
    g(\Pi) =  g(G_1) + g(G_2) + 1 =: \h_0(G).\label{eq-handle}
  \end{equation}

  \case{$W_1, W_2, W_3, W_4$ consist of three distinct faces}
  We may assume that $W_3 = W_4 = (x, e_3, U_3, f_4, y, f_3, U_4, e_4)$.
  The same construction as in the previous case (with $W_1$ and $W_2$ expressed as above) combines $W_1, W_2$, and $W_3$ into a single $\Pi$-face
  $(x, e_1, U_1, e_2, e_3, U_3, f_4, y, f_1, U_2, f_2, y, f_3, U_4, e_4, x)$.
  Again, the total number of faces decreases by two and the genus of $\Pi$ is given by~(\ref{eq-handle}).

  \case{$W_1 = W_2$ and $W_3 = W_4$}
  Observe that since $W_1 = W_2$, we have that $\theta(G_1) = 0$ and, similarly, we have $\theta(G_2) = 0$.
  Write $W_1 = W_2 = (x, e_1, U_1, f_2, y, f_1, U_2, e_2)$ and $W_3 = W_4 = (x, e_3, U_3, f_4, y, f_3, U_4, e_4)$.
  The above construction combines $W_1$ and $W_3$ into the $\Pi$-faces $(x, e_1, U_1, f_2, y, f_3, U_4, e_4)$ and $(y, f_1, U_2, e_2, x, e_3, U_3, f_4)$.
  Thus, the total number of faces did not change and~(\ref{eq-euler}) gives the following value of $g(\Pi)$.
  \begin{equation}
    g(\Pi) =  g(G_1) + g(G_2).\label{eq-same-face}
  \end{equation}
\end{casesblock}

\begin{figure}
  \centering
  \includegraphics{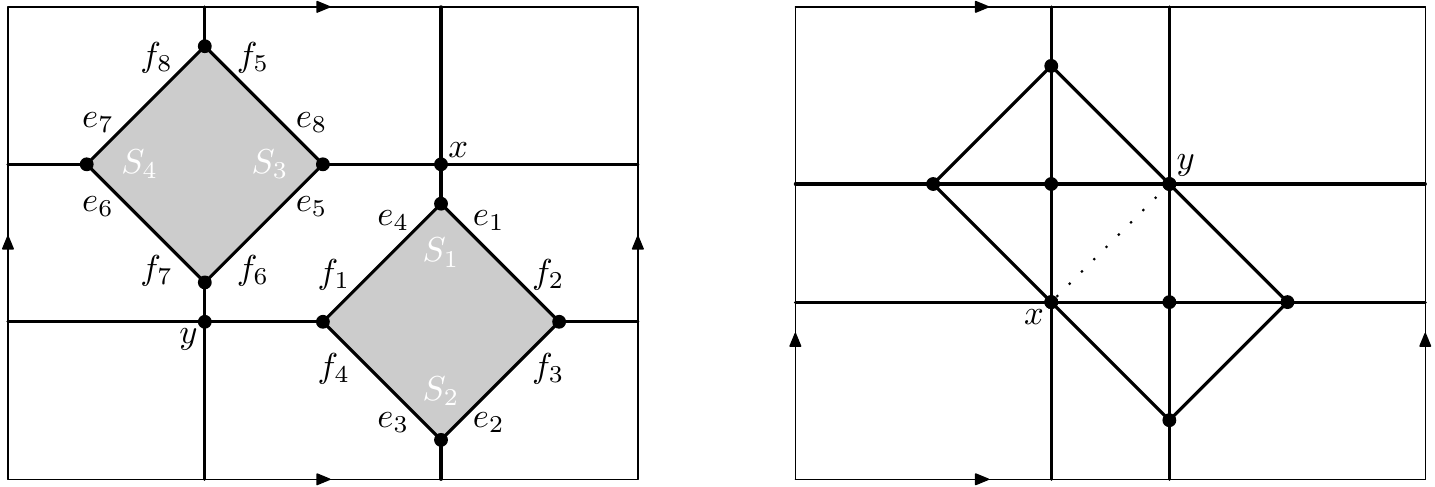}
  \caption{(a) An illustration of an embedding of an $xy$-sum of two $xy$-alternating graphs on the torus.
    For better clarity, the vertices $x$ and $y$ were split into 5 vertices each. Contract all the edges incident with $x$ and $y$ to get
    the $xy$-sum.
    (b) The 2-sum of two copies of $K_5$ embedded into the torus.}
  \label{fg-2-k5}
\end{figure}
Suppose that $\Pi_1$ and $\Pi_2$ are minimum genus embeddings of $G_1$ and $G_2$ (respectively) that are both $xy$-alternating.
Let $W_1$ and $W_2$ be the $xy$-alternating faces of $\Pi_1$ and $\Pi_2$, respectively, 
and write $W_1$ as $(x, e_1, U_1, f_2, y, f_1, U_2, e_4, x, e_3, U_3, f_4, y, f_3, U_4, e_2)$ and 
$W_2$ as $(x, e_5, U_5, f_6, y, f_5, U_6, e_8, x, e_7, U_7, f_8, y, f_7, U_8, e_6)$.
Again, the local rotation $\Pi(v)$ of a non-terminal vertex $v \in V(G_i)$ is set to $\Pi_i(v)$, $i = 1,2$.
To construct $\Pi(x)$, cut $\Pi_1(x)$ at $e_1, e_2$ and $e_3, e_4$ to obtain two linear sequences $e_1, S_1,e_4$ and $e_3,S_2,e_2$ 
and cut $\Pi_2(x)$ at $e_5,e_6$ and $e_7,e_8$ to obtain $e_5,S_3,e_8$ and $e_7,S_4,e_6$.
Let $\Pi(x)$ be the cyclic sequence $(e_1,S_1, e_4,e_5,S_3, e_8, e_3, S_2, e_2, e_7, S_4, e_6)$.
We construct $\Pi(y)$ similarly. Fig.~\ref{fg-2-k5} illustrates this process and gives an example of a 2-sum of two $K_5$'s.
The faces $W_1$ and $W_2$ are combined into $\Pi$-faces $(x, e_1, U_1, f_2, y, f_7, U_8, e_6)$, $(y, f_1, U_2, e_4, x, e_5, U_5, f_6)$,
$(x, e_3, U_3, f_4, y, f_5, U_6, e_8)$, and $(y, f_3, U_4, e_2, x, e_7, U_7, f_8)$.
As the total number of faces increased by two, (\ref{eq-euler}) gives the following value of $g(\Pi)$.
\begin{equation}
  g(\Pi) =  g(G_1) + g(G_2) - 1.\label{eq-alternating}
\end{equation}

Usually, there is a minimum genus embedding of $G$ constructed from the minimum genus embeddings of $G_1$ and $G_2$.
Suppose now that $\theta(G_1) = 1$, $\ep(G_1) = 1$, and $\e(G_2) = 1$. Since $\theta(G_1) = 1$, the only embedding described above
that we can construct from minimum genus embeddings of $G_1$ and $G_2$ has genus $g(G_1) + g(G_2) + 1$.
On the other hand, $\g(G_1^+) = g(G_1) + 1$ and both $G_1^+$ and $G_2$ are $xy$-alternating.
Thus we obtain an embedding of $G$ of genus $g(G_1^+) + g(G_2) - 1 = g(G_1) + g(G_2) < g(G_1)  + g(G_2) + 1$.
Hence it is necessary to consider also the embeddings of $G_1^+$ and $G_2^+$.
The minimum of the genera given by equations~(\ref{eq-same-face}) and~(\ref{eq-alternating}) can be combined into a single value, denoted $\h_1(G)$:
\begin{equation}
  \h_1(G) = \gp(G_1) + \gp(G_2) - \ep(G_1)\ep(G_2).\label{eq-face}
\end{equation}

Using the parameters defined above, we can write
\begin{equation*}
\h_1(G) = g(G_1) + g(G_2) + \theta(G_1) + \theta(G_2) - \ep(G_1)\ep(G_2).
\end{equation*}
The similarity of the above equation to~(\ref{eq-handle}) leads us to define the graph parameter $\eta(G_1, G_2) = \theta(G_1) + \theta(G_2) - \ep(G_1)\ep(G_2)$.
Note that $\eta(G_1, G_2) \in \{-1, 0, 1, 2\}$.
This gives another expression for $\h_1$:
\begin{equation}
\h_1(G) = g(G_1) + g(G_2) + \eta(G_1, G_2).\label{eq-eta}
\end{equation}

Decker et al.~\cite{decker-1985} proved the following formula for the genus of a 2-sum of graphs.

\begin{theorem}[Decker, Glover, and Huneke~\cite{decker-1985}]
  \label{th-decker-ori}
  Let $G$ be an $xy$-sum of connected graphs $G_1$ and $G_2$. 
  If $xy \not\in E(G)$, then $g(G)$ is the minimum of $\h_0(G)$ and $\h_1(G)$, else $g(G) = \h_1(G)$.
  Furthermore,
  \begin{enumerate}[label=\rm(\roman*)]
  \item
    $\ep(G) = 1$ if and only if $\ep(G_1) \not= \ep(G_2)$, and
  \item
    $\theta(G) = 1$ if and only if $xy \not\in E(G)$ and $\eta(G_1, G_2) = 2$.
  \end{enumerate}
\end{theorem}

Often, we consider minor-operations in the graph $G_1$ while the graph $G_2$ is fixed.
When $\ep(G_2) = 1$, the genus of $G$ depends on the graph parameter $\galt = g - \epsilon$, called the \df{alternating genus} of $G$.
Let $\gap = \gp - \ep$ be the graph parameter defined as $\gap(G) = \galt(G^+) = \gp(G) - \ep(G)$.
If we know the value of the parameter $\ep(G_2)$, then we can express $\h_1(G)$ as follows.
If $\ep(G_2) = 1$, then~(\ref{eq-face}) can be rewritten as
\begin{equation}
\h_1(G) = \gap(G_1) + \gp(G_2).\label{eq-ori-galt}
\end{equation}
Else, (\ref{eq-face}) is equivalent to
\begin{equation}
\h_1(G) = \gp(G_1) + \gp(G_2).\label{eq-ori-plus}
\end{equation}

The next lemma shows that alternating genus is a minor-monotone graph parameter.

\begin{lemma}
\label{lm-galt-minor}
  Let $G \in \G_{xy}$. If $H$ is a minor of $G$, then $\galt(H) \le \galt(G)$.
\end{lemma}

\begin{proof}
  If $g(H) < g(G)$ or $\epsilon(H) \ge \epsilon(G)$, then the result trivially holds.
  Hence if the claimed inequality is violated, then $g(H) = g(G)$, $\epsilon(H) = 0$, and $\epsilon(G) = 1$.
  Thus, there is an $xy$-alternating minimum genus embedding $\Pi$ of $G$. Let $W_a$ be an $xy$-alternating $\Pi$-face.

  We may assume without loss of generality that $H$ is obtained from $G$ by a single minor-operation.
  Suppose first that $H = G - e$ for some edge $e \in E(G)$.
  Let $\Pi'$ be the embedding of $H$ induced by $\Pi$.
  If $e$ is a singular edge that appears in a $\Pi$-face $W$, then $W$ is split into two $\Pi'$-faces in $\Pi'$.
  Thus $g(H) \le g(\Pi') = g(\Pi) - 1 = g(G) - 1$ which contradicts the assumption that $g(H) = g(G)$.
  Hence $e$ appears in two different $\Pi$-faces $W_1$ and $W_2$.
  The faces $W_1$ and $W_2$ combine to form a single $\Pi'$-face $W'$ in $\Pi'$.
  Thus $g(\Pi') = g(\Pi)$. As either $W_a$ is a $\Pi'$-face or $W_a - e$ is a subsequence of $W'$,
  we conclude that $\Pi'$ is also $xy$-alternating. This contradicts the assumption that $g(H) = g(G)$ and $\e(H) = 0$.
  
  Suppose now that $H = G/e$ for some edge $e \in E(G)$.
  Let $\Pi'$ be the induced embedding of $H$ obtained from $\Pi$ by contracting $e$.
  That is, the local rotation $\Pi'(v_e)$ around the vertex $v_e$ obtained by contraction of $e = uv$
  is set to be the concatenation of the linear sequences obtained from $\Pi(u)$ and $\Pi(v)$ by cutting them at $e$.
  If $e$ does not appear in $W_a$, then $W_a$ is also a $\Pi'$-face.
  Otherwise, as $e \not= xy$, $\Pi'$ contains a facial walk $W_a'$ that can be obtained from $W_a$ by replacing each (of at most 2) occurrence of $u, e, v$
  by $v_e$. It is immediate that $W_a'$ is an $xy$-alternating $\Pi'$-face. This again contradicts the choice of~$H$.
\end{proof}

The following lemma shows how the property of being $xy$-alternating can be expressed in terms of $\theta(G)$ and $\e(G^+)$.

\begin{lemma}
  \label{lm-alt-equiv}
  Let $G \in \Gcxy$.
  The graph $G$ is $xy$-alternating if and only if $\theta(G) = 0$ and $G^+$ is $xy$-alternating.
  In symbols, $\epsilon(G) = 1$ if and only if $\theta(G) = 0$ and $\ep(G) = 1$.
\end{lemma}

\begin{proof}
  Assume that $G$ is $xy$-alternating and let $\Pi$ be an $xy$-alternating embedding of $G$ of genus $g(G)$.
  By embedding the edge $xy$ into the $xy$-alternating $\Pi$-face, we obtain 
  an embedding of $G^+$ into the same surface that is also $xy$-alternating.
  This shows that $\theta(G) = 0$ and $\ep(G) = 1$.

  For the converse, assume that $\theta(G) = 0$ and that $G^+$ is $xy$-alternating.
  Let $\Pi$ be an $xy$-alternating embedding of $G^+$ with an $xy$-alternating $\Pi$-face $W$. 
  Since $\theta(G) = 0$, the edge $xy$ is not a singular edge.
  Thus by deleting $xy$ from $\Pi$, we obtain an embedding $\Pi'$ of $G$ in the same surface where either $W$ is a $\Pi'$-face
  or $W - xy$ is a subsequence of a $\Pi'$-face. Hence $\Pi'$ is an $xy$-alternating embedding of $G$.
  Since $g(\Pi') = g(G)$, the graph $G$ is $xy$-alternating.
\end{proof}

Fig.~\ref{fg-poset} shows the relationship between the parameters $\g, \gp, \galt$, and $\gap$. In addition to the constraints given
in the figure, there is one more interrelationship that is described by the following lemma.

\begin{lemma}
  \label{lm-galt-constraint}
  For a graph $G \in \G_{xy}$, we have either $\galt(G) = g(G)$ or $\galt(G) = \gap(G)$.
\end{lemma}

\begin{proof}  
  If $\epsilon(G) = 0$, then $\galt(G) = g(G)$ and we are done.
  Otherwise, $\e(G) = 1$ and Lemma~\ref{lm-alt-equiv} gives that $\e^+(G) = 1$ and $\theta(G) = 0$.
  Therefore, $\gap(G) = \galt(G) + \e(G) + \theta(G) - \ep(G) = \galt(G)$.
\end{proof}

\begin{figure}
  \centering
  \includegraphics{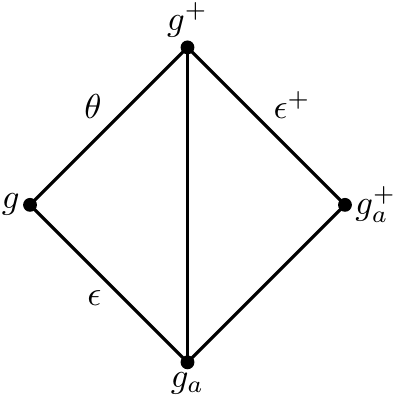}
  \caption{Hasse diagram showing relations of several graph parameters. An edge indicates that the values of parameters differ by at most one
    and the parameter below is bounded from above by the parameter above.}
  \label{fg-poset}
\end{figure}
For a graph parameter $\P$, we say that a minor-operation $\mu \in \M(G)$ \df{decreases} $\P$ by $k$
if $\P(\mu G) \le \P(G) - k$. The subset of $\M(G)$ that decreases $\P$ by $k$ is denoted by $\dc k {\P,G}$.
We write just $\dc k \P$  when the graph is clear from the context.

We shall show that each minor-operation in a 2-connected minor-tight part of an $xy$-sum decreases at least one of the graph parameters $g$, $\gp$, and $\gap$.
Note that several parameters can be decreased by a single minor-operation and it depends on the relations between the parameters.
For example, if $G$ is $K_{3,3}$ with the terminals that are non-adjacent and we consider an edge $e$ of $G$,
then the contraction $(e, /)$ belongs both to $\dc1 g$ and $\dc1 \gp$ as $g(G / e) = \gp(G / e) = 0$.
But $G$ is $xy$-alternating (see Fig.~\ref{fg-kuratowski-alt}), so $\galt(G) = \gap(G) = 0$ and $(e, /)$ belongs neither to $\dc1 \galt$ nor to $\dc1 \gap$.

Two graph parameters $\P$ and $\Q$ are \df{$1$-separated} (in this order) if
$\P(G) \le \Q(G) \le \P(G) + 1$ for all graphs $G \in \G_{xy}$.
If $\L = \Q - \P$, then we also say that $\P$ and $\Q$ are 1-separated by $\L$.
If $\P$ and $\Q$ are 1-separated by $\L$, then it is easy to see that the following holds for each $k$ and each $G \in \G_{xy}$:
\begin{enumerate}[label=\rm(S\arabic*)]
\item 
  If $\L(G) = 1$, then $\dc k \P \ss \dc k \Q$;
\item 
  if $\L(G) = 0$, then $\dc k \Q \ss \dc k \P$;
\item 
  $\dc{k+1} \P \ss \dc k \Q$ and
  $\dc{k+1} \Q \ss \dc k \P$.
\end{enumerate}

Several graph parameters defined above are 1-separated (see~Fig.~\ref{fg-poset}).
The parameters $g$ and $\gp$ are 1-separated by the parameter $\theta$, 
$\galt$ and $g$ by $\epsilon$, $\gap$ and $\gp$ by $\epsilon^+$,
and $\galt$ and $\gap$ are 1-separated by the parameter $\epsilon + \theta - \epsilon^+$.
We shall prove formally only that $\galt$ and $\gp$ are 1-separated by $\epsilon + \theta$.
As $\gp = \galt + \e + \theta$ it is enough to show the following.

\begin{lemma}
\label{lm-galt-gp-sep}
For any graph $G \in \G_{xy}$, we have that
$$\gp(G) \le \galt(G) + 1.$$
\end{lemma}

\begin{proof}
  By Lemma~\ref{lm-galt-constraint},
  either $\galt(G) = \g(G)$ or $\galt(G) = \gap(G)$.
  In the former case, $\gp(G) = \g(G) + \theta(G) \le \galt(G) + 1$.
  In the latter case, $\gp(G) = \gap(G) + \ep(G) \le \galt(G) + 1$.
\end{proof}

Using the new notation we can state the following corollary of Lemma~\ref{lm-galt-constraint}.

\begin{corollary}
\label{cr-galt}
 For each $G \in \G_{xy}$,  $\dc1 \galt \ss \dc1 g \cup \dc1 \gap$.
\end{corollary}

\begin{proof}
  Let $\mu \in \dc1{\galt}$. 
  If $\mu \not\in \dc1 g \cup \dc1 \gap$, then $g(\mu G) = g(G) > \galt(\mu G)$ and $\gap(\mu G) = \gap(G) > \galt(\mu G)$,
  which contradicts Lemma~\ref{lm-galt-constraint} (for the graph $\mu G$).
\end{proof}

\begin{table}
  \centering
  \begin{tabular}{|c | c | c | c|}
    \hline
    $xy \in E(G)$ & $\ep(G_2)$ & $\eta(G_1, G_2)$ & $\mu$ \\
    \hline
    \multirow{2}{*}{yes}  & 0 & \multirow{2}{*}{---} & $\dc1 \gp$ \\
    & 1 & & $\dc1 \gap$ \\
    \hline
    \multirow{7}{*}{no} & \multirow{3}{*}{0} &  0 & $\dc1 \gp$ \\
     & &  1 & $\dc1 g$ or $\dc1 \gp$ \\
     & &  2 & $\dc1 g$ \\
    \cline{2-4}
    & \multirow{4}{*}{1} & -1 & $\dc1 \gap$\\
    & & 0 & $\dc2 g$ or $\dc1 \gap$ \\
    & & 1 & $\dc1 g$ or $\dc1 \gap$ \\
    & & 2 & $\dc1 g$ or $\dc2 \gap$ \\
    \hline
  \end{tabular}
  \caption{Possible results for a minor-operation in a minor-tight side of a 2-sum of graphs.}
  \label{tb-side-options}
\end{table}

The next lemma describes necessary and sufficient conditions for a single part of a 2-sum of graphs to be minor-tight.
This is a key lemma and its outcome, summarized in Table~\ref{tb-side-options}, will be used heavily throughout this paper.

\begin{lemma}
  \label{lm-sides-mu}
  Let $G$ be an $xy$-sum of connected graphs $G_1$ and $G_2$ and $\mu \in \M(G_1)$ such that $\mu G_1$ is connected.
  Then $g(\mu G) < g(G)$ if and only if the following is true (where $\dc k\cdot$ always refer to the decrease of the parameter in $G_1$):
  \begin{enumerate}[label=\rm(\roman*)]
  \item
    If $xy \in E(G)$, then $\mu \in \dc1 \gp$ if $\ep(G_2) = 0$ and $\mu \in \dc1 \gap$ if $\ep(G_2) = 1$.
  \item
    If $xy \not\in E(G)$ and $\eta(G_1, G_2) = -1$, then $\mu \in \dc1 \gap$.
  \item
    If $xy \not\in E(G)$ and $\eta(G_1, G_2) = 0$, then
    $\mu \in \dc1 \gp$ when $\ep(G_2) = 0$ and
    $\mu \in \dc2 g \cup \dc1 \gap$ when $\ep(G_2) = 1$.
  \item
    If $xy \not\in E(G)$ and $\eta(G_1, G_2) = 1$, then
    $\mu \in \dc1 g \cup \dc1 \gp$ when $\ep(G_2) = 0$ and
    $\mu \in \dc1 g \cup \dc1 \gap$ when $\ep(G_2) = 1$.
  \item
    If $xy \not\in E(G)$ and $\eta(G_1, G_2) = 2$, then
    $\mu \in \dc1 g$ when $\ep(G_2) = 0$ and
    $\mu \in \dc1 g \cup \dc2 \gap$ when $\ep(G_2) = 1$.
  \end{enumerate}
\end{lemma}

\begin{proof}
  Let us start with the ``only if'' part. 
  Since $\mu G_1$ is connected, Theorem~\ref{th-decker-ori} can be used
  to determine $g(\mu G)$.
  In order to show (i), suppose that $xy \in E(G)$.
  By Theorem~\ref{th-decker-ori}, $g(G)$ and $g(\mu G)$ are equal to $\h_1(G)$
  and $\h_1(\mu G)$, respectively.
  If $\ep(G_2) = 0$, then by~(\ref{eq-ori-plus}),
  $$g^+(\mu G_1) + \gp(G_2) = g(\mu G) < g(G)= \gp(G_1) + \gp(G_2).$$
  Thus $g^+(\mu G_1) < g^+(G_1)$, yielding that $\mu \in \dc1 \gp$.
  If $\ep(G_2) = 1$, then by~(\ref{eq-ori-galt}),
  $$\gap(\mu G_1) + \gp(G_2) = g(\mu G) < g(G) = \gap(G_1) + \gp(G_2).$$
  Thus $\gap(\mu G_1) < \gap(G_1)$, yielding that $\mu \in \dc1 \gap$.

  Assume now that $xy \not\in E(G)$. 
  We will do the cases (ii), (iii) and (iv) together. Assume that $\eta(G_1, G_2) \le 1$.
  If $\ep(G_2) = 0$, let us assume that $\mu \not\in \dc1 \gp$
  and if $\ep(G_2) = 1$, let us assume that $\mu \not\in \dc1 \gap$.
  By~(\ref{eq-ori-plus}) and (\ref{eq-ori-galt}), $\h_1(\mu G) = \h_1(G)$.
  By Theorem~\ref{th-decker-ori}, $g(\mu G) = \h_0(G) < g(G)$.
  By using the definition of $\h_0(G)$ in~(\ref{eq-handle}), we obtain:
  $$\h_0(G) = g(\mu G_1) + g(G_2) + 1 = g(\mu G) < g(G) = g(G_1) + g(G_2) + \eta(G_1, G_2) .$$
  Thus $g(\mu G_1) \le g(G_1) + \eta(G_1, G_2) - 2$.
  If $\eta(G_1, G_2) = -1$, then $\mu \in \dc3 g$ which implies that $\mu \in \dc2 \gp$ by (S3) as $\g$ and $\gp$ are 1-separated.
  By another application of (S3), we obtain that $\mu \in \dc1 \gap$ yielding (ii).
  If $\eta(G_1, G_2) = 0$, then $\mu \in \dc2 g$.
  This proves (iii) when $\ep(G_2) = 1$.
  If $\ep(G_2) = 0$, then also $\mu \in \dc1 \gp$ by (S3).
  This yields (iii).
  If $\eta(G_1, G_2) = 1$, then $\mu \in \dc1 g$ which yields (iv).

  Suppose now that $\eta(G_1, G_2) = 2$ and that $\mu \not\in \dc1 g$. Then $\h_0(G) = \h_0(\mu G)$.
  By Theorem~\ref{th-decker-ori} and~(\ref{eq-eta}), $g(G) = \h_0(G)$.
  Since $g(\mu G) < g(G)$, we conclude that $g(\mu G) = \h_1(\mu G) < g(G)$.
  As $\eta(G_1, G_2) = 2$, we know that $\theta(G_1) = \theta(G_2) = 1$ and $\ep(G_1)\ep(G_2) = 0$.
  Thus we may write 
  $$g(G) = \h_0(G) = g(G_1) + g(G_2) + 1 = \gp(G_1) + \gp(G_2) - 1.$$
  If $\ep(G_2) = 0$, then we obtain using~(\ref{eq-ori-plus}) that
  $$\gp(\mu G_1) + \gp(G_2) = g(\mu G) < g(G) = \gp(G_1) + \gp(G_2) - 1.$$
  Hence $\mu \in \dc2 \gp$ which implies by (S3) that also $\mu \in \dc1 g$, a contradiction.

  If $\ep(G_2) = 1$, then $\ep(G_1) = 0$ and $\gap(G_1) = \gp(G_1)$.
  We use~(\ref{eq-ori-galt}) to obtain that
  $$\gap(\mu G_1) + \gp(G_2) = g(\mu G) < g(G) = \gap(G_1) + \gp(G_2) - 1.$$
  Hence $\mu \in \dc2 \gap$.
  This finishes the ``only if'' part.

  To prove the ``if'' part, we assume that (i)--(v) hold and show that $g(\mu G) < g(G)$.
  We start by proving that if $\mu \in \dc1 g$, $xy \not\in E(G)$, and $\eta(G_1, G_2) \ge 1$, then $g(\mu G) < g(G)$.
  By Theorem~\ref{th-decker-ori}, $g(G) = \h_0(G)$.
  Since $g(\mu G) \le \h_0(\mu G)$, we obtain that
  $$g(\mu G) \le g(\mu G_1) + g(G_2) + 1 < g(G_1) + g(G_2) + 1 = g(G).$$

  If  $\mu \in \dc2 g$, $xy \not\in E(G)$, and $\eta(G_2, G_2) = 0$, we have a similar inequality:
  $$g(\mu G) \le g(\mu G_1) + g(G_2) + 1 < g(G_1) + g(G_2) = g(G).$$

  Similarly, we do the cases when $\mu \in \dc1 \gp$ and when $\mu \in \dc1 \gap$.
  Suppose that $\mu \in \dc1 \gp$, $\ep(G_2) = 0$, and $xy \in E(G)$ or $\eta(G_1, G_2) \le 1$.
  By Theorem~\ref{th-decker-ori}, $g(G) = \h_1(G)$.
  We obtain from Theorem~\ref{th-decker-ori} and~(\ref{eq-ori-plus}) that
  $$g(\mu G) \le \gp(\mu G_1) + \gp(G_2) < \gp(G_1) + \gp(G_2) = g(G).$$

  Suppose now that $\mu \in \dc1 \gap$, $\ep(G_2) = 1$, and $xy \in E(G)$ or $\eta(G_1, G_2) \le 1$.
  We obtain from Theorem~\ref{th-decker-ori} and~(\ref{eq-ori-galt}) that
  $$g(\mu G) \le \gap(\mu G_1) + \gp(G_2) < \gap(G_1) + \gp(G_2) = g(G).$$

  In the remaining case, when $xy \not\in E(G)$, $\eta(G_2, G_2) = 2$, $\ep(G_2)=1$, and $\mu \in \dc2 \gap$, we have a similar inequality:
  $$g(\mu G) \le \gap(\mu G_1) + \gp(G_2) < \gap(G_1) + \gp(G_2) - 1 \le g(G_1) + g(G_2) + 1 = g(G).$$
  This finishes the proof of the lemma.
\end{proof}

Since for each graph precisely one hypothesis in the cases (i)--(v) of Lemma~\ref{lm-sides-mu} holds, we obtain the following corollary.

\begin{corollary}
\label{cr-sides}
  Let $G$ be an $xy$-sum of connected graphs $G_1$ and $G_2$ and $\mu \in \M(G_1)$ such that $\mu G_1$ is connected and $g(\mu G) < g(G)$.
  Then $\mu \in \dc1 g \cup \dc1 \gp \cup \dc1 \gap$.
  Furthermore, if $\ep(G_2) = 0$, then $\mu \in \dc1 g \cup \dc1 \gp$.
\end{corollary}

Lemma~\ref{lm-sides-mu} characterizes when a graph with two terminals is a part of an obstruction for an orientable surface.
The next lemma describes when the edge $xy$ is minor-tight in an $xy$-sum of graphs.

\begin{lemma}
\label{lm-edge-xy}
  Let $G$ be an $xy$-sum of connected graphs $G_1$ and $G_2$.
  If $xy \in E(G)$, then the subgraph of $G$ induced by the edge $xy$ is minor-tight 
  if and only if $\eta(G_1, G_2) = 2$ and either $g(G_1 / xy) < \gp(G_1)$ or $g(G_2 / xy) < \gp(G_2)$.
\end{lemma}

\begin{proof}
  By Theorem~\ref{th-decker-ori}(ii), $\theta(G - xy) = 1$ if and only if $\eta(G_1, G_2) = 2$.
  Thus $g(G - xy) < g(G)$ if and only if $\eta(G_1, G_2) = 2$.  We may thus assume that $\eta(G_1, G_2) = 2$. 

  Theorem~\ref{th-battle-ori} implies that
  \[
  g(G / xy) = g(G_1 / xy) + g(G_2 /xy).
  \]
  Since $\ep(G_1)\ep(G_2) = 0$, Theorem~\ref{th-decker-ori} and~(\ref{eq-face}) gives that
  $$g(G) = \h_1(G) = \gp(G_1) + \gp(G_2).$$
  Therefore, $g(G /xy) < g(G)$ if and only if $g(G_1 / xy) + g(G_2 /xy) < \gp(G_1) + \gp(G_2)$.
  Since $g(G_1 /xy) \le \gp(G_1)$ and $g(G_2 /xy) \le \gp(G_2)$, we obtain that
  $g(G /xy) < g(G)$ if and only if $g(G_1 / xy) < \gp(G_1)$ or $g(G_2 / xy) < \gp(G_2)$.
\end{proof}

\section{Critical classes for graph parameters}
\label{sc-classes}

Lemma~\ref{lm-sides-mu} provides necessary and sufficient conditions on the parts of an $xy$-sum for being minor-tight.
In this section, we shall study and categorize graphs that satisfy these conditions.

For a graph parameter $\P$, let $\C(\P)$ denote the family of graphs $G \in \G_{xy}$ such that 
each minor-operation in $G$ decreases $\P$ by at least 1, i.e., $\M(G) = \dc1 \P$.
We call $\C(\P)$ the \df{critical class} for $\P$.
Let $\Cc(\P)$ be the subfamily of $\C(\P)$ of graphs without the edge $xy$.
We refine the class $\C(\P)$ according to the value of $\P$. 
Let $\C_k(\P)$ denote the subfamily of $\C(\P)$ that contains precisely the graphs $G$ for which $\P(G) = k+1$.
The classes $\Cc_k(\P)$ are defined similarly as subfamilies of $\Cc(\P)$.

In this section, we shall study the classes $\Cc(g)$, $\Cc(\gp)$, $\Cc(\galt)$, and $\Cc(\gap)$.
It is easy to see that, for each graph $G \in \Cc_k(g)$, the graph $\hat{G}$ is an obstruction for $\SS_k$.
On the other hand, for each graph $G \in \Forb(\SS_k)$ and two non-adjacent vertices $x$ and $y$ of $G$,
the graph in $\G_{xy}$ obtained from $G$ by making $x$ and $y$ terminals belongs to $\Cc_k(g)$.
Similarly to $\Cc_k(g)$, the family $\Cc_k(\gp)$ can be constructed from the graphs in $\Forb(\SS_k)$.

We shall denote by $\Forb^*(\SS)$ the class of \df{minimal forbidden topological minors} for the surface $\SS$.
The graphs in $\Forb^*(\SS)$ are the minimal graphs (of minimum degree at least 3)  with respect to deletion of edges that are not embeddable into $\SS$.

\begin{lemma}
\label{lm-gp}
Let $G \in \Cc_k(\gp)$. If $\theta(G) = 0$, then $\hat{G} \in \Forb(\SS_k)$.
If $\theta(G) = 1$, then either $\hat{G^+} \in \Forb(\SS_k)$, or $\hat{G^+} \in \Forb^*(\SS_k)$ and $\hat{G / xy} \in \Forb(\SS_k)$.
\end{lemma}

\begin{proof}
  If $\theta(G) = 0$, then $\M(G) = \dc1 g$ by (S2) and thus $G \in \Cc_k(\g)$.
  Therefore $\hat{G} \in \Forb(\SS_k)$ as explained above.
  Suppose now that $\theta(G) = 1$.
  Since $G \in \Cc(\gp)$, $\M(G) \ss \dc1 {g,G^+}$.
  As $g(G^+ - xy) < g(G^+)$ (and all other minor-operations in $\hat{G^+}$ except contracting the edge $xy$ decrease the genus of $G^+$), 
  we have that $\hat{G^+} \in \Forb^*(\SS_k)$.
  If $g(G / xy) < \gp(G)$, then $\hat{G^+} \in \Forb(\SS_k)$ since both deletion and contraction of $xy$ decreases the genus of $G^+$.
  On the other hand, if $g(G /xy) = \gp(G)$, take any minor-operation $\mu \in \M(G /xy)$.
  Since $\mu$ is also a minor-operation in $G$, we obtain that
  $g(\mu(G /xy)) \le \gp(\mu G) < \gp(G) = g(G / xy)$ as $\mu(G /xy)$ is a minor of $\hat{\mu G^+}$.
  Since $\mu$ was chosen arbitrarily, $G/xy \in \Forb(\SS_k)$.
\end{proof}

Since the parameters $g$ and $\gp$ are 1-separated, the graphs whose minor-operations decrease either $g$ or $\gp$
belong to either $\Cc(g)$ or $\Cc(\gp)$.

\begin{lemma}
\label{lm-theta}
  Let $G \in \Gcxy$.
  If $\M(G) = \dc1 g \cup \dc1 \gp$, then $G$ belongs to either $\Cc(g)$ or $\Cc(\gp)$.
\end{lemma}

\begin{proof}
  If $\theta(G) = 0$, then $\dc1 \gp \ss \dc1 g$ by (S2). Thus $\M(G) = \dc1 g$ and $G \in \Cc(g)$.
  Similarly, if $\theta(G) = 1$, then $\dc1 g \ss \dc1 \gp$ by (S1).
  We conclude that $\M(G) = \dc1 \gp$ and $G \in \Cc(\gp)$.
\end{proof}

The classes $\Cc(\galt)$ and $\Cc(\gap)$ are related to the class $\C(\galt)$ which was introduced by Mohar and \v{S}koda~\cite{mohar-obstr},
who proved that the classes $\C_k(\galt)$ are finite (for each $k \ge 1$).
By the following lemma, this implies that both $\Cc_k(\galt)$ and $\Cc_k(\gap)$ are finite.
Observe that a graph $G \in \Gcxy$ belongs to $\C(\galt)$ if and only if it belongs to $\Cc(\galt)$.
The graphs in $\C(\galt) \sm \Cc(\galt)$ can be characterized as follows.

\begin{lemma}
\label{lm-gap-minus-galt}
For a graph $G \in \Gcxy$ and $k\ge 0$, we have that $G^+ \in \C_k(\galt)$ if and only if $G \in \Cc_k(\gap) \sm \Cc_k(\galt)$.
\end{lemma}

\begin{proof}
  Suppose that $G^+ \in \C_k(\galt)$. It is immediate that $G \in \Cc_k(\gap)$. 
  Since $\galt(G) = \galt(G^+ - xy) < \galt(G^+) = k + 1$, the graph $G$ does not belong to $\Cc_k(\galt)$.

  Suppose now that $G \in \Cc_k(\gap) \sm \Cc_k(\galt)$.
  If $\galt(G) = \gap(G)$, then $\M(G) = \dc1 \galt$ by (S2) and it follows that $G \in \Cc_k(\galt)$.
  Thus $\galt(G) < \gap(G)$. Hence $\galt(G^+) > \galt(G) = \galt(G^+ - xy)$ and $(xy, -) \in \dc1 {\galt, G^+}$.
  We conclude that $G^+ \in \C_k(\galt)$ as $\galt(G^+) = \gap(G) = k + 1$.
\end{proof}

Also the graphs that do not belong to $\Cc(\gap)$ can be characterized.

\begin{lemma}
  \label{lm-galt-minus-gap}
  If $G \in \Cc(\galt)$, then $G \not\in \Cc(\gap)$ if and only if there exists $\mu \in \M(G)$ such that $\mu \in \dc1 g \sm \dc1 \gap$.
\end{lemma}

\begin{proof}
  The ``if'' part follows from the fact that $\M(G) \not= \dc1 \gap$.
  The ``only if'' part follows from Corollary~\ref{cr-galt} as there is $\mu \in \M(G)$ such that $\mu \not\in \dc1 \gap$.
\end{proof}

Corollary~\ref{cr-galt} says that each minor-operation that decreases alternating genus also decreases $g$ or $\gap$.
We have the following weakly converse statement.

\begin{lemma}
\label{lm-g-or-gap}
  Let $G \in \Gcxy$.
  If $\M(G) = \dc1 g \cup \dc1 \gap$, then $G$ belongs to at least one of $\Cc(g)$, $\Cc(\galt)$, or $\Cc(\gap)$.
\end{lemma}

\begin{proof}
  By Lemma~\ref{lm-galt-constraint}, either $\galt(G) = g(G)$ or $\galt(G) = \gap(G)$.
  If $\g(G) = \gap(G) = \galt(G)$, then $\dc1 g \ss \dc1 \galt$ and $\dc1 \gap \ss \dc1 \galt$ by (S2).
  Thus $G \in \Cc(\galt)$.

  If $\g(G) > \galt(G)$, then $\dc1 \gap \ss \dc1 \galt$ by (S2). 
  By (S1), $\dc1 \galt \ss \dc1 \g$. We conclude that $G \in \Cc(g)$.
  Similarly, if $\gap(G) > \galt(G)$, then $\dc1 \g \ss \dc1 \galt$ by (S2).
  By (S1), $\dc1 \galt \ss \dc1 \gap$. We conclude that $G \in \Cc(\gap)$.
\end{proof}



\section{Hoppers}
\label{sc-hoppers}

In this section, we describe three subfamilies of $\Cc(\gp)$ all of which we call \df{hoppers}. 
If $G$ is a graph in $\Cc(\gap)$ such that $\ep(G) = 1$, then 
we call $G$ a \df{hopper of level 2\/}. 
It is immediate from (S1) that $G \in \Cc(\gp)$.
The level should indicate the difficulty to construct such a graph.
Hoppers of level 0 and 1 appear as parts of obstructions of
connectivity 2 and are defined below.

A graph $G \in \Gcxy$ is a \df{hopper of level 1\/} if $\M(G) = \dc1 \gap \cup \dc2 g$ and $G \not\in \Cc(\gap)$.
Similarly, a graph $G \in \Gcxy$ is a \df{hopper of level 0\/} if $\M(G) = \dc1 \g \cup \dc2 \gap$ and $G \not\in \Cc(\g)$.
Let $\H^l$, $0 \le l \le 2$, denote the family of hoppers of level $l$.
Let $\H_k^l$ denote the subfamily of $\H^l$ of graphs $G$ with $\gp(G) = k$.

\begin{lemma}
\label{lm-hopper-0}
  If $G \in \H^0$, then $G \in \Cc(\gp)$, $\ep(G) = 0$, and $\theta(G) = 1$.
\end{lemma}

\begin{proof}
  By (S3), $\dc2 \gap \ss \dc1 \gp$.
  If $\theta(G) = 0$, then $\dc1 \gp \ss \dc1 g$ by (S2) --- a contradiction with $G \not\in \Cc(g)$.
  Hence $\theta(G) = 1$.
  By (S1), $\dc1 g \ss \dc1 \gp$ and we conclude that $G \in \Cc(\gp)$.

  If $\ep(G) = 1$, then $\dc2 \gap \ss \dc2 \gp$ by (S1) and, since $\dc2 \gp \ss \dc1 \g$ by (S3), we have that $\dc2 \gap \ss \dc1 \g$, a contradiction.
  Thus $\ep(G) = 0$.
\end{proof}

Note that the proof of the next lemma is analogous to the proof of Lemma~\ref{lm-hopper-0}.

\begin{lemma}
\label{lm-hopper-1}
  If $G \in \H^1$, then $G \in \Cc(\gp)$, $\ep(G) = 1$, and $\theta(G) = 0$.
\end{lemma}

\begin{proof}
  By (S3), $\dc2 \g \ss \dc1 \gp$.
  If $\ep(G) = 0$, then $\dc1 \gp \ss \dc1 \gap$ by (S2) --- a contradiction with $G \not\in \Cc(\gap)$.
  Hence $\ep(G) = 1$.
  By (S1), $\dc1 \gap \ss \dc1 \gp$ and we conclude that $G \in \Cc(\gp)$.

  If $\theta(G) = 1$, then $\dc2 \g \ss \dc2 \gp$ by (S1) and, since $\dc2 \gp \ss \dc1 \gap$ by (S3), we have that $\dc2 \g \ss \dc1 \gap$, a contradiction.
  Thus $\theta(G) = 0$.
\end{proof}

Similarly to the genus, alternating genus decreases by at most 1 when an edge is deleted.

\begin{lemma}
\label{lm-galt-deletion}
  Let $G \in \G_{xy}$. For each $e \in E(G)$,
  $\galt(G - e) \ge \galt(G) - 1$.
\end{lemma}

\begin{proof}
  Suppose that $\galt(G - e) < \galt(G) - 1$.
  Since $g(G - e) \ge g(G) - 1$, we have that $\epsilon(G) = 0$, $\epsilon(G - e) = 1$, and $g(G-e) = g(G) - 1$.
  Let $\Pi$ be an $xy$-alternating embedding of $G - e$ in $\SS_k$, $k = g(G-e)$ and 
  let $W$ be an $xy$-alternating $\Pi$-face.
  If the endvertices $u$ and $v$ of $e$ are $\Pi$-cofacial, then $\Pi$ can be extended to an embedding of $G$ in $\SS_k$, a contradiction.
  Otherwise, let $\Pi'$ be the embedding of $G$ on $\SS_{k+1}$ obtained from $\Pi$ by embedding $e$ onto a new handle connecting faces incident with $u$ and $v$.
  Since $W$ is a subwalk of a $\Pi'$-face, $\Pi'$ is $xy$-alternating. Since $g(\Pi') = g(G - e) + 1 = g(G)$, we have that $\epsilon(G) = 1$ which is a contradiction.
\end{proof}

Lemma~\ref{lm-galt-deletion} has the following corollary that shows the motivation for introducing the notion of hoppers of level 2.

\begin{corollary}
\label{cr-hoppers}
A graph $G \in \Cc_k(\gap)$ does not embed into $\SS_{k+1}$ if and only if $\ep(G) = 1$.
\end{corollary}

Mohar and \v{S}koda conjectured that all graphs in $\C_k(\galt)$ embed into $\SS_{k+1}$.

\begin{conjecture}[Mohar and \v{S}koda~\cite{mohar-obstr}]
\label{cj-same-surface}
  Each $G \in \C_k(\galt)$ embeds into $\SS_{k+1}$.
\end{conjecture}

We suspect that there are no hoppers of level 1 and 2.

\begin{conjecture}
\label{cj-no-hoppers}
  There are no hoppers of level 1 and 2.
\end{conjecture}

Thus Conjecture~\ref{cj-no-hoppers} is a stronger version of a Conjecture~\ref{cj-same-surface}.
The following lemma shows that Conjecture~\ref{cj-same-surface} is true if $xy \in E(G)$.

\begin{lemma}
  \label{lm-alt-edge}
  Let $G \in \Gcxy$.
  Then $\galt(G) < \gap(G)$ if and only if $\ep(G) = 0$ and $\theta(G) = 1$.
\end{lemma}

\begin{proof}
  If $\galt(G) < \gap(G)$, then $\galt(G) = \g(G)$ by Lemma~\ref{lm-galt-constraint}.
  Since $\galt$ and $\gap$ are 1-separated,  $\gap(G) - \galt(G) = 1 = \e(G) + \theta(G) - \ep(G)$.
  Since $\e(G) = 0$, we obtain that $\theta(G) = 1$ and $\ep(G) = 0$ as required.

  If $\ep(G) = 0$ and $\theta(G) = 1$, then $\e(G) = 0$ by Lemma~\ref{lm-alt-equiv}.
  Thus $\galt(G) < \galt(G) + \e(G) + \theta(G) - \ep(G) = \gap(G)$.
\end{proof}

Lemmas~\ref{lm-gap-minus-galt} and~\ref{lm-alt-edge} assert that a hopper of level 2 belongs to the class $\Cc_k(\galt)$.

\section{Dumbbells}

Lemma~\ref{lm-sides-mu} provides information about all minor-operations in a part of an $xy$-sum except for a deletion of an edge that disconnects the graph.
Since the $xy$-sum is 2-connected, deletion of such a cut-edge of $G_1$ separates $x$ and $y$.
In this section, we determine how minor-tight parts of an $xy$-sum with such a cut-edge look like.

If $G_1 \in \Gcxy$ and $b \in E(G_1)$ is a cut-edge of $G_1$ whose deletion separates $x$ and $y$, 
we say that $G_1$ is a \df{dumbbell} with \df{bar} $b$.

\begin{lemma}
  \label{lm-dumbbell-ep}
  If $G_1$ is a dumbbell with bar $b$, then $\ep(G_1) = 0$ and $(b, /) \not\in \dc1 g \cup \dc1 \gp$.
\end{lemma}

\begin{proof}
  Suppose for a contradiction that $\ep(G_1) = 1$; then there exists an $xy$-alternating minimum-genus embedding $\Pi$ of $G_1^+$.
  Let $W$ be an $xy$-alternating $\Pi$-facial walk.
  The walk $W$ can be split into 4 subwalks containing $x$ and $y$.
  Each of the edges $xy$ and $b$ appears precisely twice in the $\Pi$-facial walks 
  (either once in two different $\Pi$-facial walks or twice in a single $\Pi$-facial walk).
  Since each walk from $x$ to $y$ has to use either $xy$ or $b$, both $xy$ and $b$ are singular edges that
  appear twice in $W$. Since $\Pi$ is an orientable embedding, the edge $xy$ appears in $W$ once in the direction from $x$ to $y$
  and once from $y$ to $x$. Hence, there is another appearance of one of the terminals, say $x$, in $W$ that is not incident with the edge $xy$.
  We can write $W$ as $W = (x,xy,y,W_1,e_1,x,e_2,W_2,y,xy,x,e_3,W_3,e_4,x)$. The local rotation around $x$ can be written as $(xy, e_4, S_1, e_2, e_1, S_2, e_3)$.
  Let $\Pi'$ be the embedding obtained from $\Pi$ by letting $\Pi'(v) = \Pi(v)$ for $v \in V(G_1) \sm \{x\}$
  and $\Pi(x) = (e_4, S_1, e_2, xy, e_1, S_2, e_3)$. All $\Pi$-facial walks except $W$ are also $\Pi'$-facial walks as all $\Pi$-angles not incident with $W$
  are also $\Pi'$-angles. The $\Pi$-facial walk $W$ is split into three $\Pi'$-facial walks: $(x,xy,y,W_1,e_1,x)$, $(x, e_3,W_3,e_4,x)$, 
  and $(x, e_2, W_2, y, xy, x)$. Thus $g(\Pi') < g(\Pi)$, a contradiction with $\Pi$ being a minimum-genus embedding of $G_1^+$.
  We conclude that $\ep(G_1) = 0$.

  Let $\mu = (b, /)$ be the contraction operation of $b$ in $G_1$. We shall show that $\mu \not\in \dc1 g \cup \dc1 \gp$.
  Let $H_1$ and $H_2$ be the components of $G_1 - b$.
  By Theorem~\ref{th-battle-ori}, $g(G_1) = g(H_1) + g(H_2) = g(\mu G_1)$.
  If $b$ is incident with a terminal, say $b = zy$, $z \in V(H_1)$, then
  $G_1^+$ is the 1-sum of $H_1 + b + xy$ and $H_2$.
  By Theorem~\ref{th-battle-ori}, 
  $$g(G_1^+) = g(H_1 + b + xy) + g(H_2) = g(H_1 + xz) + g(H_2) = g(\mu G_1^+).$$
  Thus $g^+(G_1) = \gp(\mu G_1)$.

  Suppose that $b$ is not incident with a terminal and let $z \in V(H_1)$ be an endpoint of $b$. 
  Consider the graphs $H_1' = H_1 + xy$ and $H_2' = H_2 + b$ as members of the class $\G_{yz}^\circ$.
  Observe that $H_1'$ and $H_2'$ are dumbbells (in $\G_{yz}^\circ$).
  We have already shown that $\ep(H_1') = \ep(H_2') = 0$
  and $g(\mu H_2') = g(H_2')$ and, since the bar of $H_2'$ is incident with a terminal, 
  $\gp(\mu H_2') = \gp(H_2')$.
  By Theorem~\ref{th-decker-ori} (when $G_1^+$ is viewed as a $yz$-sum of $H_1'$ and $H_2'$),
  \begin{align*}
    g(G_1^+) &= \min\{g(H_1') + g(H_2') +1, \gp(H_1') + \gp(H_2')\}, \text{ and} \\
    g(\mu G_1^+) &= \min\{g(H_1') + g(\mu H_2') +1, \gp(H_1') + \gp(\mu H_2')\}. 
  \end{align*}
  Since $g(\mu H_2') = g(H_2')$ and $\gp(\mu H_2') = \gp(H_2')$, we conclude that $g(\mu G_1^+) = \g(G_1^+)$. 
  Thus $\gp(\mu G_1) = \gp(G_1)$.
  This shows that $\mu \not\in \dc1 g \cup \dc1 \gp$.
\end{proof}

\begin{lemma}
  \label{lm-dumbbell-unique}
  Let $G$ be an $xy$-sum of connected graphs $G_1$ and $G_2$.
  If $G_1$ is a dumbbell with bar $b$ and $G_1$ is minor-tight in $G$, 
  then $\ep(G_1 / b) = 1$ and $b$ is unique, that is, $G_1$ has a single cut-edge separating $x$ and $y$.
\end{lemma}

\begin{proof}
  By Lemma~\ref{lm-dumbbell-ep} and Corollary~\ref{cr-sides}, $(b, /) \in \dc1 \gap \sm \dc1 \gp$.
  It is immediate that $\ep(G_1 / b) = 1$.

  For the second part, suppose that there is another bar $e \not= b$ in $G_1$.
  By Lemma~\ref{lm-dumbbell-ep}, $\ep(G_1 / b) = 0$ as $G_1/b$ is a dumbbell with bar $e$, a contradiction.
  We conclude that $b$ is unique.
\end{proof}

Let $\D$ be the class of dumbbells $G_1$ with bar $b$ such that $\theta(G_1) = 0$, $\mu \in \dc1 g$ for each $\mu \in \M(G_1) \sm \{(b, -), (b, /)\}$,
and $\ep(G_1 / b) = 1$. 

\begin{lemma}
  \label{lm-dumbbell}
  Let $G$ be an $xy$-sum of connected graphs $G_1$ and $G_2$ such that $G_1$ is a dumbbell.
  Then $G_1$ is minor-tight in $G$ if and only if $\ep(G_2) = 1$ and one of the following holds:
  \begin{enumerate}[label=\rm(\roman*)]
  \item
    $G_1 \in \Cc(\gap) \sm \Cc(\galt)$, $\theta(G_1) = 1$, and either $xy \in E(G)$ or $\eta(G_1, G_2) = 1$.
  \item
    $G_1 \in \D$, $xy \not\in E(G)$, and $\eta(G_1, G_2) = 1$.
  \end{enumerate}
\end{lemma}

\begin{proof}
  Assume that $G_1$ is minor-tight in $G$.
  By Lemmas~\ref{lm-dumbbell-ep} and~\ref{lm-dumbbell-unique}, $\ep(G_1) = 0$ and $G_1$ has a unique bar $b$ for which it holds that $(b, /) \not\in \dc1 g \cup \dc1 \gp$ and $\ep(G_1 / b) = 1$.
  Hence $\gap(G_1/b) = \gap(G_1) - 1$ and we have that $(b, /) \in \dc1 \gap \sm \dc2 \gap$.
  By Corollary~\ref{cr-sides}, $\ep(G_2) = 1$.

  Assume first that $\theta(G_1) = 1$. We shall show that (i) holds. 
  If $xy \not\in E(G)$ and $\eta(G_1, G_2) = 2$, then $(b, /)$ violates Lemma~\ref{lm-sides-mu} as $(b, /) \not\in \dc1 g \cup \dc2 \gap$.
  Thus either $xy \in E(G)$ or $\eta(G_1, G_2) \le 1$.
  Since $\ep(G_1) = 0$ and $\theta(G_1) = 1$, we conclude that either $xy \in E(G)$ or $\eta(G_1, G_2) = 1$.

  Since $g(G_1 - b) = g(G_1)$ and $\e(G_1 - b) = 0$ (as the terminals of $G_1 - b$ are not connected), $(b, -) \not\in \dc1 \galt$.
  Hence $G_1 \not\in \Cc(\galt)$. It remains to show that $G_1 \in \Cc(\gap)$.

  Since $\theta(G_1) = 1$, we have that $\gp(G_1 - b) = g(G_1^+ - b) = g(G_1) < \gp(G_1)$ and thus $(b, -) \in \dc1 \gp$.
  By Lemma~\ref{lm-dumbbell-ep}, $\ep(G_1) = 0$. By (S2), $(b, -) \in \dc1 \gap$.

  Let $\mu \in \M(G_1) \sm \{(b,-), (b,/)\}$. Since $\mu G_1$ is connected, Lemma~\ref{lm-sides-mu} gives that $\mu \in \dc1 \gap$ if $xy \in E(G)$
  and $\mu \in \dc1 g \cup \dc1 \gap$ if $xy \not\in E(G)$ and $\eta(G_1, G_2) = 1$.
  By (S1), $\dc1 g \ss \dc1 \gp$. By (S2), $\dc1 \gp \ss \dc1 \gap$. We conclude that $\mu \in \dc1 \gap$.
  Since $\mu$ was arbitrary and $(b,-),(b,/) \in \dc1 \gap$, we have that $\M(G_1) = \dc1 \gap$ and $G_1 \in \Cc(\gap)$.
  Therefore, (i) holds.

  Assume now that $\theta(G_1) = 0$. We shall show that (ii) holds. In $G - b$, the two components of $G_1 - b$  are joined to $G_2$ by single vertices.
  If $xy \in E(G)$, Theorems~\ref{th-battle-ori} and~\ref{th-decker-ori} imply (using $\ep(G_1) = 0$ and $\theta(G_1) = 0$) that
  $$g(G - b) = g(\hat{G_1 - b}) + g(\hat{G_2^+})  = \g(G_1) + \gp(G_2)  = \gp(G_1) + \gp(G_2)  = \h_1(G) = g(G).$$ 
  This contradicts the assumption that $G_1$ is minor-tight. We conclude that $xy \not\in E(G)$.
  If $\eta(G_2) = 0$, we obtain a similar contradiction.
  $$g(G - b) = g(\hat{G_1 - b}) + g(\hat{G_2}) = \g(G_1) + g(G_2) = \gp(G_1) + \gp(G_2) = \h_1(G) = g(G).$$ 
  Thus $\eta(G_1, G_2) \ge 1$. Since $\theta(G_1) = 0$, we conclude that $\theta(G_2) = 1$ and $\eta(G_1, G_2) = 1$.

  It remains to show that $G_1 \in \D$, namely that $\mu \in \dc1 g$ for each $\mu \in \M(G) \sm \{(b, -), (b, /)\}$.
  Let $\mu \in \M(G) \sm \{(b, -), (b, /)\}$. Since $\mu G_1$ is connected, $\mu \in \dc1 g \cup \dc1 \gap$ by Lemma~\ref{lm-sides-mu}.
  Since $\mu G$ is still a dumbbell, $\ep(\mu G) = 0$ by Lemma~\ref{lm-dumbbell-ep}. Hence $\gp(\mu G) = \gap(\mu G)$
  and $\dc1 \gap \ss \dc1 \gp$. By (S2), $\dc1 \gp \ss \dc1 g$. Therefore, $\mu \in \dc1 g$.
  We conclude that $G_1 \in \D$. Thus (ii) holds.

  Let us prove the ``if'' part of the theorem. 
  Assume that $\ep(G_2) = 1$ and that (i) holds.
  Let $\mu \in \M(G_1)$. We have that $\mu \in \dc1 \gap$.
  If $\mu G_1$ is connected, $g(\mu G) < g(G)$ by Lemma~\ref{lm-sides-mu} since $\ep(G_2) = 1$.
  Otherwise, $\mu = (b, -)$.
  If $xy \in E(G)$, then by Theorems~\ref{th-battle-ori} and~\ref{th-decker-ori},
  $$g(G - b) = g(\hat{G_1 - b}) + g(\hat{G_2^+}) = g(G_1) + \gp(G_2) < \gp(G_1) + \gp(G_2) = \h_1(G) = g(G).$$
  If $xy \not\in E(G)$ and $\eta(G_1, G_2) = 1$, then $\theta(G_2) = 0$ and we obtain that
  $$g(G - b) = g(\hat{G_1 - b}) + g(\hat{G_2}) = g(G_1) + g(G_2) < g(G_1) + g(G_2) + 1 = g(G).$$
  In both cases $g(G - b) < g(G)$ and thus $g(\mu G) < g(G)$ for each $\mu \in \M(G_1)$. 
  We conclude that $G_1$ is minor-tight in $G$.

  Assume now that (ii) holds.
  Let $\mu \in \M(G_1)$ and assume first that $\mu G_1$ is connected.
  If $\mu = (b, /)$, then $(b, /) \in \dc1 \gap$ since $\ep(G_1) = 0$ and $\ep(\mu G_1) = 1$ (and $\gp(\mu G_1) = \gp(G_1)$).
  Otherwise $\mu \in \dc1 g$ since $G_1 \in \D$.
  Since $xy \not\in E(G)$, $\eta(G_1, G_2) = 1$, and $\ep(G_2) = 1$, Lemma~\ref{lm-sides-mu} gives that $g(\mu G) < g(G)$.

  The case when $\mu = (b, -)$ remains.
  By Theorems~\ref{th-battle-ori} and~\ref{th-decker-ori},
  $$g(G - b) = g(\hat{G_1 - b}) + g(\hat{G_2}) = g(G_1) + g(G_2) < g(G_1) + \g(G_2) +1 = g(G).$$
  We have that $g(\mu G) < g(G)$ for each $\mu \in \M(G_1)$.
  We conclude that $G_1$ is minor-tight.
\end{proof}

We close this section by showing that, in an obstruction of connectivity 2, there always exists a 2-vertex-cut
such that neither of the parts belongs to $\D$.

\begin{lemma}
  \label{lm-dumbbell-no}
  Let $G \in \Forb(\SS_k)$ be of connectivity 2. Then there exists a 2-vertex-cut $\{x, y\}$ such that
  neither of the parts of $G$ when viewed as an $xy$-sum of two graphs belongs to $\D$.
\end{lemma}

\begin{proof}
  \begin{figure}
    \centering
    \includegraphics{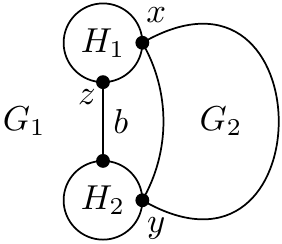}
    \caption{A sketch of the structure of the graph $G$ from the proof of Lemma~\ref{lm-dumbbell-no}.}
    \label{fg-dumbbell-sketch}
  \end{figure}
  Let $G$ be an $xy$-sum of $G_1$ and $G_2$. Suppose that $G_1 \in \D$.
  Since $G_1$ is minor-tight in $G$ and $\theta(G_1) = 0$, Lemma~\ref{lm-dumbbell} gives that $\ep(G_2) = 1$ and $\eta(G_1, G_2) = 1$.
  From the definition of $\eta(G_1, G_2)$ we conclude that $\theta(G_2) = 1$ as $\ep(G_1) = 0$.
  Let $b$ be a bar of $G_1$ and let $H_1$ and $H_2$ be the components of $G_1 - b$.
  We may assume that $H_1$ contains at least one edge. Let $x$ be the common vertex of $H_1$ and $G_2$ and let $z$
  be the endpoint of $b$ incident with $H_1$. Let us view $G$ as an $xz$-sum of $H_1$ and $G_2' = G_2 + H_2 + b$ (see Fig.~\ref{fg-dumbbell-sketch}).
  We claim that neither $H_1$ nor $G_2'$ belongs to $\D$.

  By Lemma~\ref{lm-dumbbell-unique} applied to $G_1$, $b$ is the unique cut-edge separating $x$ and $y$ and thus
  there is no cut-edge in $H_1$ separating $x$ and $z$. Therefore, $H_1$ is not a dumbbell.
  We shall show that $\theta(G_2') = 1$ and hence $G_2' \not\in \D$.
  The graph $G_2'^+$ can be viewed as an $xy$-sum of $G_2$ and the graph $G_1' = H_2 + b + zx$.
  The graph $G_1'$ is a dumbbell and thus $\ep(G_1') = 0$ by Lemma~\ref{lm-dumbbell-ep}.
  By Theorem~\ref{th-battle-ori},
  $g(G_2') = g(H_2) + g(G_2)$.
  By Theorem~\ref{th-decker-ori}, using $\ep(G_1) = 0$ and $\theta(G_2) = 1$, 
  $$g(G_2'^+) = \min\{g(G_1') + g(G_2) + 1, g(G_1'^+) + g(G_2^+)\} \ge g(H_2) + g(G_2) + 1.$$
  Therefore $\theta(G_2') = 1$. We conclude that $G_2' \not\in \D$.
\end{proof}


\section{General orientable surface}

In this section, we prove a general theorem that classifies minor-tight parts of a 2-sum of graphs.
The classification that is given in Theorem~\ref{th-general} below is also summarized in Table~\ref{tb-side-classes}.

\begin{table}
  \centering
  \begin{tabular}{|c | c | c | c|}
    \hline
    $xy \in E(G)$ & $\ep(G_2)$ & $\eta(G_1, G_2)$ & $G_1$ belongs to \\
    \hline
    \multirow{2}{*}{yes}  & 0 & \multirow{2}{*}{---} & $\Cc(\gp)$ \\
    & 1 & & $\Cc(\gap)$ \\
    \hline
    \multirow{7}{*}{no} & \multirow{3}{*}{0} &  0 & $\Cc(\gp)$ \\
     & &  1 & $\Cc(g)$ or $\Cc(\gp)$ \\
     & &  2 & $\Cc(g)$ \\
    \cline{2-4}
    & \multirow{4}{*}{1} & -1 & $\Cc(\gap)$ \\
    & & 0 & $\Cc(\gap)$ or $\H^1$ \\
    & & 1 & $\Cc(g)$, $\Cc(\galt)$, $\Cc(\gap)$, or $\D$ \\
    & & 2 & $\Cc(\g)$ or $\H^0$ \\
    \hline
  \end{tabular}
  \caption{Classification of minor-tight parts of 2-sums of graphs.}
  \label{tb-side-classes}
\end{table}

\begin{theorem}
  \label{th-general}
  Let $G$ be an $xy$-sum of connected graphs $G_1$ and $G_2$. The graph $G_1$ is minor-tight if and only if
  the following statements hold (see Table~\ref{tb-side-classes}).
  \begin{enumerate}[label=\rm(\roman*)]
  \item 
    If $xy \in E(G)$, then $G_1 \in \Cc(\gp)$ if $\ep(G_2) = 0$ and $G_1 \in \Cc(\gap)$ otherwise.  
  \item
    If $xy \not\in E(G)$ and $\eta(G_1, G_2) = -1$, then $G_1 \in \Cc(\gap)$.  
  \item
    If $xy \not\in E(G)$ and $\eta(G_1, G_2) = 0$, then $G_1 \in \Cc(\gp)$ if $\ep(G_2) = 0$ and $G_1 \in \Cc(\gap) \cup \H^1$ otherwise.  
  \item
    If $xy \not\in E(G)$ and $\eta(G_1, G_2) = 1$, then $G_1 \in \Cc(g) \cup \Cc(\gp)$ if $\ep(G_2) = 0$ and $G_1 \in \Cc(g) \cup \Cc(\galt) \cup \Cc(\gap) \cup \D$ otherwise.
  \item
    If $xy \not\in E(G)$ and $\eta(G_1, G_2) = 2$, then $G_1 \in \Cc(g)$ if $\ep(G_2) = 0$ and $G_1 \in \Cc(\g) \cup \H^0$ otherwise.
  \end{enumerate}
\end{theorem}

\begin{proof}
  Let us start with the ``only if'' part of the theorem.
  Assume first that $G_1$ has no cut-edge that separates $x$ and $y$.
  Lemma~\ref{lm-sides-mu} classifies which graph parameters of $G_1$ are decreased by the minor-operations in $\M(G_1)$.
  If it is a single parameter, $G_1$ belongs to the critical class corresponding to the parameter.
  For example, if $xy \in E(G)$ and $\ep(G_2) = 0$, then $\M(G_1) = \dc1 \gp$ by Lemma~\ref{lm-sides-mu}(i) and thus $G_1 \in \Cc(\gp)$.
  The statements (i), (ii), (iii) for $\ep(G_2) = 0$, and (v) for $\ep(G_2) = 0$ are proven in this way and we omit the details.
  Let us focus on the remaining cases. In all of them, we have that $xy \not\in E(G)$.

  Let us start with the case when $\eta(G_1, G_2) = 1$.
  If $\ep(G_2) = 0$, then $\M(G_1) = \dc1 g \cup \dc1 \gp$ by Lemma~\ref{lm-sides-mu}(iv).
  By Lemma~\ref{lm-theta}, $G_1$ belongs to either $\Cc(g)$ or $\Cc(\gp)$.
  If $\ep(G_2) = 1$, then $\M(G_1) = \dc1 \g \cup \dc1 \gap$ by Lemma~\ref{lm-sides-mu}(iv).
  By Lemma~\ref{lm-g-or-gap}, $G_1$ belongs to either $\Cc(g), \Cc(\galt)$, or $\Cc(\gap)$. This proves (iv).

  If $\eta(G_1, G_2) = 0$ and $\ep(G_2) = 1$, then $\M(G_1) = \dc1 \gap \cup \dc2 \g$ by Lemma~\ref{lm-sides-mu}(iii).
  By definition, $G_1$ belongs to either $\Cc(\gap)$ or $\H^1$. Thus, (iii) holds.
  If $\eta(G_1, G_2) = 2$ and $\ep(G_2) = 1$, then $\M(G_1) = \dc1 \g \cup \dc2 \gap$ by Lemma~\ref{lm-sides-mu}(v).
  By definition, $G_1$ belongs to either $\Cc(\g)$ or $\H^0$. 
  Thus (v) is true.

  Assume now that $G_1$ has a cut-edge that separates $x$ and $y$ and thus $G_1$ is a dumbbell.
  Since $G_1$ is minor-tight, Lemma~\ref{lm-dumbbell} gives that $\ep(G_2) = 1$ and that 
  (1) either $G_1 \in \Cc(\gap)$ and $xy \in E(G)$ or $\eta(G_1, G_2) = 1$, or 
  (2) $G_1 \in \D$, $xy \not\in E(G)$, and $\eta(G_1, G_2) = 1$. 
  The statements (ii), (iii), and (v) are vacuously true since either $xy \in E(G)$ or $\eta(G_1, G_2) = 1$.
  The statement (i) is true since $G_1 \in \Cc(\gap)$ if $xy \in E(G)$.
  The statement (iv) is true since either $G_1 \in \Cc(\gap)$ or $G_1 \in \D$.
  This completes the ``only if'' part of the proof.

  It remains to prove the ``if'' part. Lemma~\ref{lm-sides-mu} is now used to prove that $G_1$ is minor-tight.
  Assume first that $G_1$ has no cut-edge separating $x$ and $y$.
  If $G_1$ belongs to one of the classes $\Cc(g)$, $\Cc(\gp)$, or $\Cc(\gap)$, then it is straightforward to check that in each case
  Lemma~\ref{lm-sides-mu} asserts that $G_1$ is minor-tight. We shall omit the proof here and do only the cases when $G_1 \in \Cc(\galt)$ or $G_1$ is a hopper.

  If $G_1 \in \Cc(\galt)$, $xy \not\in E(G)$, $\ep(G_2) = 1$, and $\eta(G_1, G_2) = 1$, then Corollary~\ref{cr-galt} asserts that $\M(G_1) = \dc1 g \cup \dc1 \gap$.
  Lemma~\ref{lm-sides-mu} gives that $G_1$ is minor-tight.
  Finally, let us assume that $G_1$ is a hopper. 
  If $G_1 \in \H^1$, $xy \not\in E(G)$, $\ep(G_2) = 1$, and $\eta(G_1, G_2) = 0$, then $\M(G_1) = \dc2 g \cup \dc1 \gap$ by definition of $\H^1$.
  Lemma~\ref{lm-sides-mu} gives that $G_1$ is minor-tight.
  If $G_1 \in \H^0$, $xy \not\in E(G)$, $\ep(G_2) = 1$, and $\eta(G_1, G_2) = 2$, then $\M(G_1) = \dc1 g \cup \dc2 \gap$ by definition of $\H^0$.
  Lemma~\ref{lm-sides-mu} gives that $G_1$ is minor-tight.

  Assume now that $G_1$ is a dumbbell with bar $b$.
  If $G_1 \in \D$ (and $\ep(G_2) = 1$, $xy \not\in E(G)$, and $\eta(G_1, G_2) = 1$), then $G_1$ is minor-tight by Lemma~\ref{lm-dumbbell}(ii).
  By Lemma~\ref{lm-dumbbell-ep}, $G_1 \not\in \Cc(g) \cup \Cc(\gp)$.
  Since $\H^0$ and $\H^1$ are subsets of $\Cc(\gp)$, we have that $G_1 \not\in \H^0 \cup \H^1$ (Lemmas~\ref{lm-hopper-0} and~\ref{lm-hopper-1}).
  Thus we may assume that $G_1 \in \Cc(\galt) \cup \Cc(\gap)$ and $\ep(G_2) = 1$.
  By Lemma~\ref{lm-dumbbell-ep}, $\ep(G_1) = 0$. 
  Since $\e(G_1 - b) = \ep(G_1 - b) = 0$ and $g(G_1 - b) = g(G)$, we conclude that $G_1 \not\in \Cc(\galt)$.
  Hence $G_1 \in \Cc(\gap) \sm \Cc(\galt)$ and $(b, -) \in \dc1 \gap$.
  Since $\ep(G_1 - b) = \ep(G_1) = 0$, $(b, -) \in \dc1 \gp$.
  By (S2), $\theta(G_1) = 1$. Since $\ep(G_1) = 0$ and $\eta(G_1, G_2) \le 1$, we conclude that $\eta(G_1, G_2) = 1$.
  By Lemma~\ref{lm-dumbbell}(i), $G_1$ is minor-tight in $G$.
  This completes the proof of the theorem.
\end{proof}

Note that a graph can belong to several critical classes at the same time. 
For example, if $G \in \Cc(g)$ such that $\theta(G) = 1$ and $\ep(G) = 0$, then $G$ belongs to all four classes, $\Cc(g)$, $\Cc(\gp)$, $\Cc(\galt)$, and $\Cc(\gap)$.

We finish this section by the following corollary which shows that at least one side of a 2-sum is an ``obstruction'' for a surface.

\begin{corollary}
  \label{cr-necessary-obstr}
  Let $G$ be an $xy$-sum of connected graphs $G_1$ and $G_2$.
  If both, $G_1$ and $G_2$, are minor-tight, then the following statements hold:
  \begin{enumerate}[label=\rm(\roman*)]
  \item
    $G_1$ and $G_2$ belong to $\Cc(g) \cup \Cc(\gp) \cup \Cc(\galt) \cup \Cc(\gap) \cup \D$.
  \item
    If $\ep(G_2) = 0$, then $G_1 \in \Cc(g) \cup \Cc(\gp)$.
  \item
    Either $G_1$ or $G_2$ belongs to $\Cc(g) \cup \Cc(\gp)$.
  \end{enumerate}
\end{corollary}

\begin{proof}
  By Lemma~\ref{lm-hopper-0} and~\ref{lm-hopper-1}, $\H^0$ and $\H^1$ are subfamilies of $\Cc(\gp)$.
  Thus (i) and (ii) follow from Theorem~\ref{th-general} as it covers all possible combinations of the parameters describing $G$.
  We shall now prove (iii).
  Assume that $G_2$ does not belong to $\Cc(g) \cup \Cc(\gp)$.
  If $G_2$ is a dumbbell, then Lemma~\ref{lm-dumbbell-ep} gives that $\ep(G_2) = 0$ and thus $G_1 \in \Cc(g) \cup \Cc(\gp)$ by (ii).
  Thus we may assume that $\mu G_2$ is connected for each $\mu \in \M(G_2)$.
  Lemma~\ref{lm-theta} applied to $G_2$ gives that there exists a minor-operation $\mu \in \M(G_2)$ such that $\mu \not\in \dc1 g \cup \dc1 \gp$.
  By Corollary~\ref{cr-sides}, $\mu \in \dc1 \gap$.
  Since $\mu \not\in \dc1 \gp$, we have that $\ep(G_2) = 0$ by (S1).
  Therefore, (ii) gives that $G_1 \in \Cc(g) \cup \Cc(\gp)$.
\end{proof}


\section{Torus}
\label{sc-torus}

\begin{figure}
  \centering
  \includegraphics{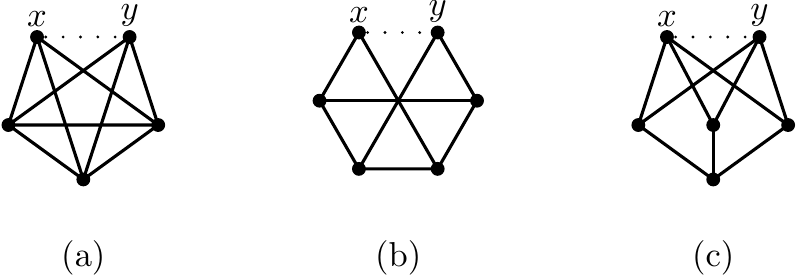}
  \caption{The family $\Cc_0(\gp)$, the third graph is the sole member of the family $\Cc_0(\g)$.}
  \label{fg-c0-gp}
\end{figure}
In this section, we characterize obstructions for embedding graphs into the torus of connectivity 2.
We first show that the classes $\Cc_0(g)$ and $\Cc_0(\gp)$ are related to Kuratowski graphs $K_5$ and $K_{3,3}$.

\begin{lemma}
\label{lm-plane-obstr}
  The class $\Cc_0(g)$ consists of a single graph, $K_{3,3}$ with non-adjacent terminals (Fig.~\ref{fg-c0-gp}(c)).
  The class $\Cc_0(\gp)$ consists of the three graphs shown in Fig.~\ref{fg-c0-gp}.
\end{lemma}

\begin{proof}
  The obstructions $\Forb(\SS_0)$ for the 2-sphere are $K_{3,3}$ and $K_5$. 
  As we observed in Sect.~\ref{sc-classes}, a graph $G$ belongs to $\Cc_0(g)$ if only if $\hat{G}$ is isomorphic to a graph in $\Forb(\SS_0)$ with the terminals non-adjacent.
  Since $xy \not\in E(G)$, $\hat{G}$ cannot be isomorphic to $K_5$, and
  there is a unique 2-labeled graph isomorphic to $K_{3,3}$ with two non-adjacent terminals. 

  Let us show first that each graph in Fig.~\ref{fg-c0-gp} belongs to $\Cc_0(\gp)$.
  If $\hat{G^+}$ is isomorphic to a Kuratowski graph, the lemma follows from the Kuratowski theorem.
  Otherwise $\hat{G}$ is isomorphic to $K_{3,3}$ with $x$ and $y$ non-adjacent.
  It suffices to show that $\mu G^+$ is planar for each minor-operation $\mu \in \M(G)$ as $G^+$ clearly embeds into the torus.
  Pick an arbitrary edge $e \in E(G)$.
  The graph $G^+ - e$ has 9 edges and is not isomorphic to $K_{3,3}$ as it contains a triangle.
  The graph $G^+ / e$ has only 5 vertices and (at most) 9 edges.
  Since $e$ was arbitrary, it follows that $\mu G^+$ is planar for every $\mu \in \M(G)$.
  We conclude that $G \in \Cc_0(\gp)$.

  We shall show now that there are no other graphs in $\Cc_0(\gp)$.
  Let $G \in \Cc_0(\gp)$. By Lemma~\ref{lm-gp}, there is a graph $H \in \Forb^*(\SS_0)$ such that either $\hat{G}$ is isomorphic to $H$
  or $G$ is isomorphic to the graph obtained from $H$ by deleting an edge and making the ends terminals.
  It is not hard to see that this yields precisely the graphs in Fig.~\ref{fg-c0-gp}.
\end{proof}

Note that the first two graphs in Fig.~\ref{fg-c0-gp} have $\theta$ equal to 1 and last one has $\theta$ equal to 0.
We summarize the properties of graphs in $\Cc_0(\gp)$ in the following lemma.

\begin{lemma}
  \label{lm-alt-kuratowski}
  For each graph $G \in \Cc_0(\gp)$, the graph $G^+$ is $xy$-alternating on the torus, $G /xy$ is planar, 
  and $\theta(G) = 1$ if and only if $G \not\in \Cc_0(g)$.
\end{lemma}

\begin{proof}
  By Lemma~\ref{lm-plane-obstr}, $\hat{G}$ or $\hat{G^+}$ is isomorphic to a Kuratowski graph.
  The $xy$-alternating embeddings of Kuratowski graphs are depicted in Fig.~\ref{fg-kuratowski-alt}.
  Since each Kuratowski graph $G$ is $xy$-alternating for each pair of vertices of $G$, 
  the graph $G^+$ is also $xy$-alternating for each pair of vertices of $G$ by Lemma~\ref{lm-alt-equiv}.
  For each Kuratowski graph $G$, the graph $G /xy$ has at most 5 vertices and at most 9 edges. 
  Thus $G/xy$ contains no Kuratowski graph as a minor and is therefore planar.
\end{proof}

\begin{figure}
  \centering
  \includegraphics{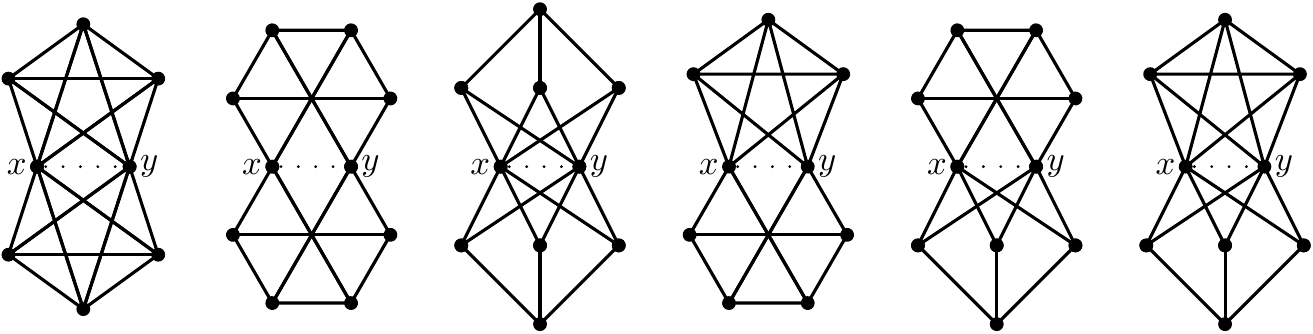}
  \caption{$\T_2$, the $xy$-sums of Kuratowski graphs which belong to $\Cc_0(\galt) \cap \Cc_0(\gap)$.}
  \label{fg-kuratowski-sum}
\end{figure}
Mohar and \v{S}koda~\cite{mohar-obstr} presented the complete list of graphs in $\C_0(\galt)$.
We describe them using six subclasses $\T_1, \ldots, \T_6$ of $\Gcxy$.
Let $\T_1$ be the class of graphs that contains each $G \in \Gcxy$ such that $\hat{G}$ is isomorphic to a Kuratowski graph plus one or two isolated vertices that are terminals in $G$,
$\T_2$ the class of graphs shown in Fig.~\ref{fg-kuratowski-sum},
$\T_3$ the class of graphs corresponding to the graphs in Fig.~\ref{fg-alt-1-con},
$\T_4$ the class of graphs corresponding to the graphs in Fig.~\ref{fg-alt-2-con}, 
$\T_5$ the class of graphs depicted in Fig.~\ref{fg-split}, and
$\T_6$ the class of graphs corresponding to the graphs in Fig.~\ref{fg-alt-xy}.
\begin{figure}
  \centering
  \includegraphics{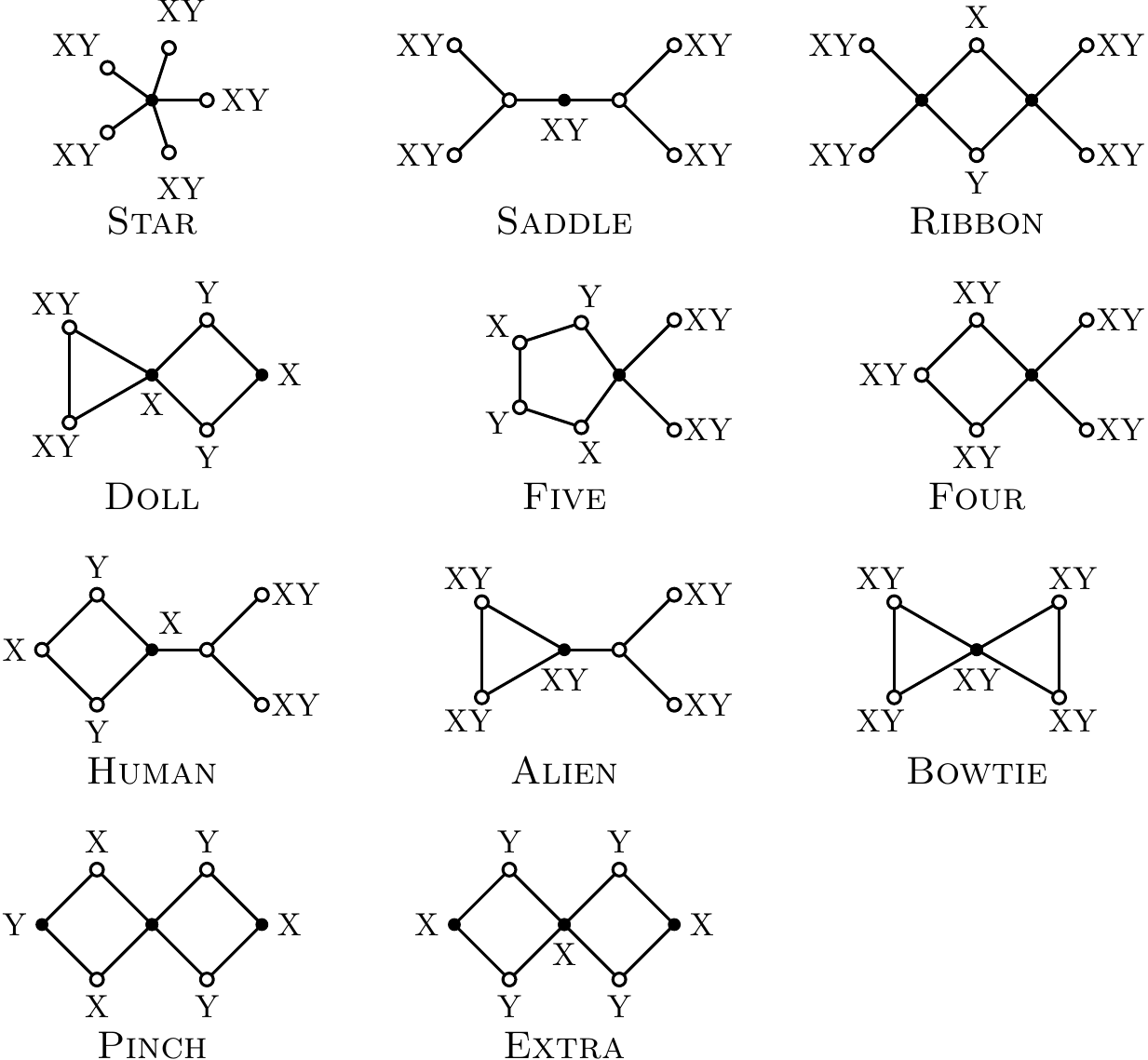}
  \caption{The XY-labelled representation of  $\T_3 \ss \Cc_0(\galt) \cap \Cc_0(\gap)$. For each white vertex $v \in V(G)$, we have $g(G - v) = 1$.}
  \label{fg-alt-1-con}
\end{figure}

\begin{theorem}[Mohar and \v{S}koda~\cite{mohar-obstr}]
\label{th-alt-obstr}
A graph $G \in \G_{xy}$ belongs to $\C_0(\galt)$ if and only if one of the following holds.
  \begin{enumerate}[label=\rm(\roman*)]
  \item
    $xy \not\in E(G)$ and $G \in \T_1 \cup \cdots \cup \T_4$.
  \item 
    $xy \in E(G)$ and $G - xy \in \T_5 \cup \T_6$.
  \end{enumerate}
\end{theorem}

The graphs in $\T_1$ are disconnected and hence they do not appear in an $xy$-sum of connectivity 2.
We will use the following facts about the class $\C_0(\galt)$.

\begin{lemma}
  \label{lm-alt-torus}
  For each graph $G \in \C_0(\galt)$, we have $\gp(G) = \g(G) = 1$ and hence $\e(G) = \ep(G) = \theta(G) = 0$.
\end{lemma}

\begin{proof}
  Observe that each graph in $\C_0(\galt)$ is nonplanar. We shall prove that $\gp(G) \le 1$ for each $G \in \T_1 \cup \cdots \cup \T_6$
  which implies that $\gp(G) = \g(G) = 1$ for each $G \in \C_0(\galt)$ by Theorem~\ref{th-alt-obstr}.
  For a graph $G \in \T_1$, $\hat{G^+}$ has two blocks, one isomorphic to a Kuratowski graph and the other consisting of a single edge.
  Thus $\gp(G) = g(\hat{G^+}) = 1$.
  Each graph $G$ in $\T_2$ can be obtained as an $xy$-sum of two Kuratowski graphs. Theorem~\ref{th-decker-ori} gives that $g(G) = 1$
  and $\theta(G) = 0$ since both parts of $G$ are $xy$-alternating. Hence $\gp(G) = 1$.

  To prove that a graph $G \in \T_3 \cup \T_4$ has $\gp(G) = 1$, it is sufficient to provide an embedding of $G^+$ in the torus.
  Fig.~\ref{fg-alt-1-con} and~\ref{fg-alt-2-con} show that $G - x - y$ has a drawing in the plane with all neighbors of $x$ and $y$ on the outer face.
  Thus $G /xy$ is a planar graph. Moreover, the edges in the local rotation around the identified vertex in $G /xy$ can be written
  as $S_1S_2\cdots S_6$ where edges in $S_1,S_3,S_5$ are those incident with $x$ in $G$ and $S_2,S_4,S_6$ are incident with $y$ in $G$.
  Therefore $G^+$ admits an embedding in the torus as shown in Fig.~\ref{fg-3-alt}.
  In the figure, a single edge is drawn from $x$ to the boundary of the planar patch for all the consecutive edges that connect $x$ and the planar patch.
  
  We shall show that this structure of graphs in $\T_3 \cup \T_4$ is not accidental. 
  Let $e \in E(G)$ be an edge incident with $x$ or $y$, say $e = xv$.
  If $G - e$ is nonplanar, then $G - e$ has an $xy$-alternating embedding $\Pi$ into the torus. 
  The two $\Pi$-angles at $x$ of the $xy$-alternating face divide the edges in the local rotation around $x$ into two sets, $S_1$ and $S_3$.
  Similarly, the edges incident with $y$ form sets $S_2$ and $S_4$.
  It is not hard to see that, since $G / xy$ is planar, we can pick $\Pi$ so that $v$ is $\Pi$-cofacial with $y$ (it is not $\Pi$-cofacial with $x$ since $G$ is not $xy$-alternating). 
  We may assume that $v$ lies in the region of edges in $S_4$.
  Thus $G / xy$ has the structure described above with $S_5 = \{e\}$ and $S_4$ split into sets $S_4'$ and $S_6$.
  It is thus enough to show that there exists an edge $e$ incident with $x$ or $y$ such that $G - e$ is nonplanar.
  For $G \in \T_3$ and an edge $e \in E(G)$ incident with a white vertex in Fig.~\ref{fg-alt-1-con}, $G - e$ is nonplanar.
  For $G \in \T_4$, the edges $e$ such that $G - e$ is nonplanar are depicted in Fig.~\ref{fg-alt-2-con} as underlined labels.
  
  Each graph $G$ in $\T_5 \cup \T_6$ is planar. Thus $\gp(G) = g(G^+) \le 1$. 
\end{proof}

\begin{figure}
  \centering
  \includegraphics{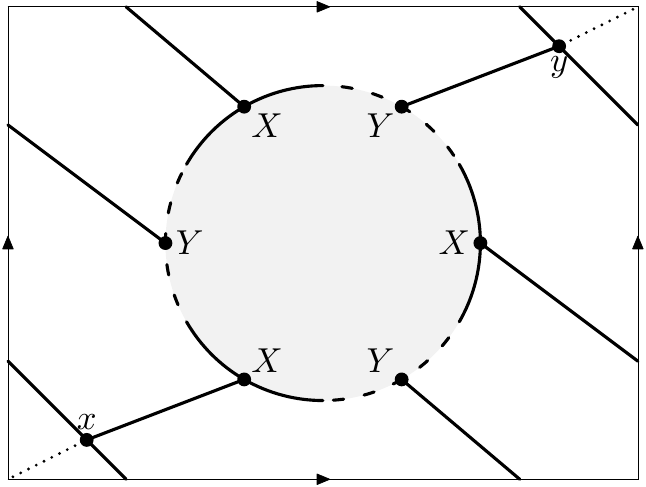}
  \caption{An embedding of $G^+$ in the torus for a graph $G$ such that $G / xy$ is planar and $G - e$ is $xy$-alternating in the torus for some edge $e$ incident with $x$ or $y$.}
  \label{fg-3-alt}
\end{figure}
We suspect that $\ep(G) = \e(G) = \theta(G) = 0$ for all graphs in $\C(\galt)$ but the proof seems out of reach. See~\cite{mohar-obstr} for more details.
Lemmas~\ref{lm-alt-edge} and~\ref{lm-alt-torus} classify when a graph in $\C_0(\galt) \cup \C_0(\gap)$ has $\theta$ equal to 1.
We have the following corollary.

\begin{corollary}
  \label{cr-theta-torus}
  Let $G$ be a graph in $\Cc_0(\galt) \cup \Cc_0(\gap)$. 
  Then $\gp(G) =1$ and $\ep(G) = 0$.
  Moreover, $\theta(G) = 1$ if and only if $G \in \Cc_0(\gap) \sm \Cc_0(\galt)$.
\end{corollary}

\begin{proof}
  Let $G \in \Cc_0(\gap) \sm \Cc_0(\galt)$.
  By Lemma~\ref{lm-gap-minus-galt}, $G^+ \in \C_0(\galt)$.
  By Lemma~\ref{lm-alt-edge}, $\theta(G) = 1$ and $\ep(G) = 0$.
  Since $\gap(G) = 1$, $\gp(G) = \gap(G) - \ep(G) = 1$.

  If $G \in \Cc_0(\galt)$, then $G \in \C_0(\galt)$ and thus $\theta(G) = \ep(G) = 0$ and $\gp(G) = 1$ by Lemma~\ref{lm-alt-torus}.
\end{proof}

\begin{figure}
  \centering
  \includegraphics{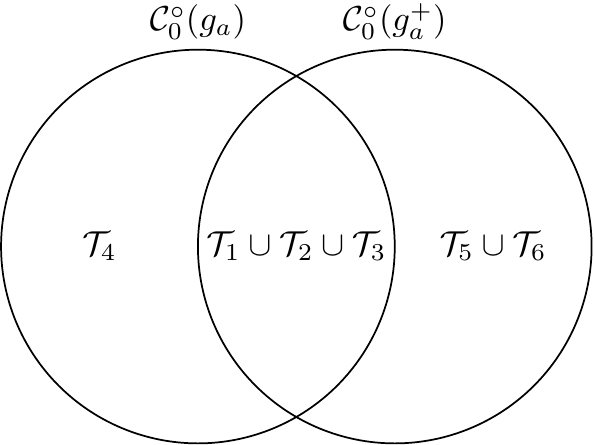}
  \caption{Venn diagram of critical classes (for alternating genus) for the torus.}
  \label{fg-cc-alt-diagram}
\end{figure}
The classes $\T_1, \ldots, \T_6$ lie in $\Cc_0(\gap) \cup \Cc_0(\galt)$. More precise membership as depicted in Fig.~\ref{fg-cc-alt-diagram} is proven below.
We shall use the following observation.

\begin{lemma}
\label{lm-vertex-cover}
  Let $G \in \G_{xy}$, $\P$ a minor-monotone graph parameter, and $v \in V(G) \sm \{x,y\}$.
  If $\P(G- v) = \P(G)$, then $\P(\mu G) = \P(G)$ for each $\mu = (uv, \cdot) \in \M(G)$.
\end{lemma}

\begin{proof}
  Let $\mu = (uv, \cdot) \in \M(G)$.
  Since $G - v$ is a minor of $\mu G$ and $\P$ is minor-monotone,
  $\P(G) \ge \P(\mu G) \ge \P(G - v) = \P(G)$. 
\end{proof}

Lemma~\ref{lm-vertex-cover} can be used to prove that $\dc1 g = \emptyset$ if we can find 
a vertex cover $U$ of $G$ such that $g(G-v) = g(G)$ for each $v \in U$.
We shall use this idea to prove that $\T_3 \ss \Cc_0(\gap)$. The following lemma will be also used.

\begin{lemma}[Lemma~19,~\cite{mohar-obstr}]
\label{lm-not-2-alt}
  Let $G \in \Gcxy$ be a graph such that $G / xy$ is planar. If $\gap(G) \ge 1$, then
  either $x$ and $y$ have at least five common neighbors or there are six distinct non-terminal vertices $v_1, \ldots, v_6$ 
  such that $v_1, v_2, v_3$ are adjacent to $x$ and $v_4, v_5, v_6$ are adjacent to $y$.
\end{lemma}

In order to determine if a graph $G \in \Cc_0(\galt)$ also belongs to $\Cc_0(\gap)$ we can either use Lemma~\ref{lm-galt-minus-gap}
or note that, since $\gap(G) \ge \galt(G)$ and $\gap$ is minor-monotone by Lemma~\ref{lm-galt-minor}, each graph $G \in \Cc_0(\galt)$ contains a graph in $\Cc_0(\gap)$ as a minor.

\begin{lemma}
  \label{lm-ti-membership}
  $\Cc_0(\galt) \cap \Cc_0(\gap) = \T_1 \cup \T_2 \cup \T_3$,
  $\Cc_0(\galt) \sm \Cc_0(\gap) = \T_4$, and
  $\Cc_0(\gap) \sm \Cc_0(\galt) = \T_5 \cup \T_6$.
\end{lemma}

\begin{proof}
  By Theorem~\ref{th-alt-obstr}, $\T_1 \cup \T_2 \cup \T_3 \cup \T_4 \ss \Cc_0(\galt)$.
  Let us start by proving that $\T_1 \cup \T_2 \cup \T_3 \ss \Cc_0(\gap)$.
  Suppose that $G \in \T_1$. Then it is not difficult to see that $G \in \Cc_0(\gap)$ since $\hat{G^+}$ has two blocks, 
  one isomorphic to a Kuratowski graph and the other consisting of a single edge.

  Let $G \in \T_2$ and $\mu \in \M(G)$.
  Since $G$ is an $xy$-sum of two graphs in $\Cc_0(\gp)$,  neither contraction nor deletion of an edge 
  on one side destroys the Kuratowski graph on the other side.
  Thus $g(\mu G) = 1$ and $\M(G) \cap \dc1 g = \emptyset$. 
  By Lemma~\ref{lm-galt-minus-gap}, $G \in \Cc_0(\gap)$.
\begin{figure}
  \centering
  \includegraphics{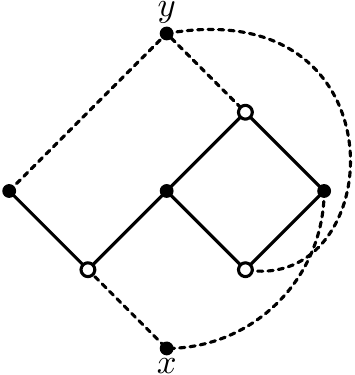}
  \caption{The graph \name{Pinch} minus a vertex. The white vertices form one part of the $K_{3,3}$-subdivision.}
  \label{fg-pinch-u}
\end{figure}

  Let us prove now that $\T_3 \ss \Cc_0(\gap)$.
  Consider a graph $G \in \T_3$.
  By Lemma~\ref{lm-galt-minus-gap}, it is enough to show that $\dc1 g \sm \dc1 \gap = \emptyset$.
  Let $\mu \in \M(G)$. 
  Let $U$ be the set of white vertices of $G$ as depicted in Fig.~\ref{fg-alt-1-con}. 
  It is not hard to show that, for each $v \in U$, $G - v$ is nonplanar.
  We omit the detailed proof of this fact and only demonstrate the proof technique on the graph \name{Pinch}.
  Since $U$ is an orbit of the isomorphism group of \name{Pinch}, it is enough to show that $G - u$ is nonplanar
  for one of the vertices $u \in U$. Indeed, $G - u$ is isomorphic to a subdivision of $K_{3,3}$ as is exhibited in Fig.~\ref{fg-pinch-u}.
  Thus $G - u$ is nonplanar for each $u \in U$ as required.

  By Lemma~\ref{lm-vertex-cover}, we may assume that the edge $e$ of $\mu$ is not covered by a vertex in $U$.
  This proves that the graphs \name{Star}, \name{Ribbon}, \name{Five} and \name{Four} are in $\Cc_0(\gap)$ since $U$ is a vertex cover.
  For the other graphs, observe that the vertices in $U$ cover all the edges not incident with a terminal. 
  Thus $e$ corresponds to a label on a black vertex of $G$ in Fig.~\ref{fg-alt-1-con}.
  Assume that $\mu = (e, -)$.
  By inspection, the conclusion of Lemma~\ref{lm-not-2-alt} is violated for $G - e$. Hence $\gap(\mu G) = 0$ and $\mu \in \dc1 \gap$.
  We may assume now that $\mu = (e, /)$. 
  When $G$ is one of the graphs \name{Saddle}, \name{Human}, \name{Alien}, \name{Bowtie}, when $G$ is \name{Extra}
  with $e$ incident with the non-terminal vertex of degree 5, and when $G$ is \name{Doll} with $e$ incident with the non-terminal vertex of degree 5, 
  the graph $\mu G^+$ is an $xy$-sum of two graphs $G_1$ and $G_2$.
  We observe that in all cases, the graphs $G_1^+$ and $G_2^+$ are planar and thus $\mu G^+$ is planar by Theorem~\ref{th-decker-ori}. We conclude that $\mu \in \dc1 \gap$.
  If $G$ is \name{Pinch}, then $\mu G$ is a proper minor of \name{Four}. Since we already showed that \name{Four} $\in \Cc_0(\gap)$,
  we have that $\mu \in \dc1 \gap$ in this case as well. If $G$ is \name{Doll} and $e$ is incident with the black vertex of degree 3, 
  then $\mu G$ is a proper minor of \name{Four}.
 The remaining case is that $G$ is \name{Extra} and $e$ is incident with a non-terminal black vertex
  of degree 3. Again, $\mu G$ is a proper minor of \name{Five} and thus $\mu \in \dc1 \gap$.

  By Lemma~\ref{lm-gap-minus-galt}, the  class $\Cc_0(\gap) \sm \Cc_0(\galt)$ contains precisely the graphs $G$ such that $G^+ \in \C_0(\galt)$.
  Theorem~\ref{th-alt-obstr} gives that the graphs in $\T_5 \cup \T_6$ (and only those) have that property.

  We prove that $\Cc_0(\galt) \sm \Cc_0(\gap) = \T_4$ by showing that $\T_4 \cap \Cc_0(\gap) = \emptyset$.
  Since each $G \in \T_4$ has a proper minor
  in $\T_6 \ss \Cc_0(\gap)$, $G$ does not belong to $\Cc_0(\gap)$. 
  \name{Pentagon} is a minor of \name{Rocket} and \name{Lollipop}, while \name{Hexagon} is
  a minor of \name{Bullet}, \name{Frog}, and \name{Hive}.
  Hence $\T_4 \ss \Cc_0(\galt) \sm \Cc_0(\gap)$.
  We have shown that the classes $\T_1, \T_2, \T_3, \T_5, \T_6$ are subclasses of $\Cc_0(\gap)$.
  Hence $\Cc_0(\galt) \sm \Cc_0(\gap) \ss \T_4$ by Theorem~\ref{th-alt-obstr}.
\end{proof}

\begin{figure}
  \centering
  \includegraphics{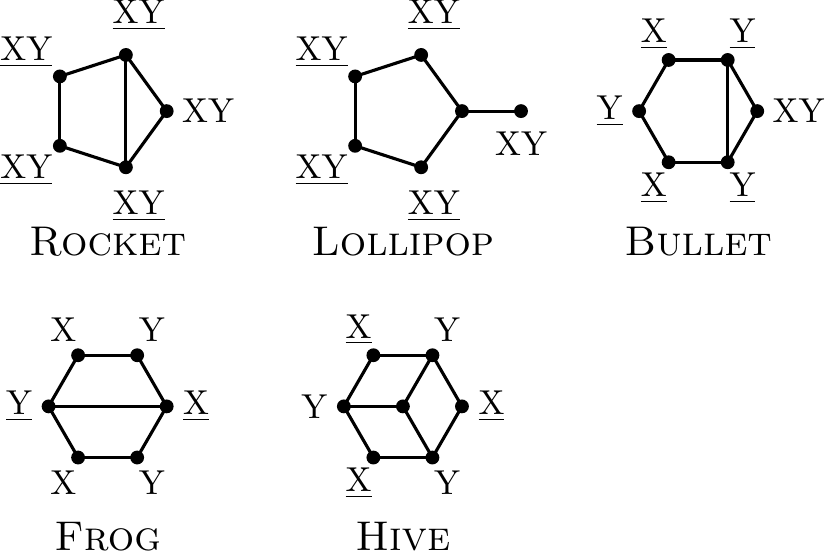}
  \caption{The XY-labelled representation of $\T_4 = \Cc_0(\galt) \sm \Cc_0(\gap)$.}
  \label{fg-alt-2-con}
\end{figure}
Let us present some restrictions on an $xy$-sum that is an obstruction for the torus.

\begin{lemma}
  \label{lm-torus-prop}
  Let $G$ be an $xy$-sum of connected graphs $G_1$ and $G_2$.
  If $G \in \Forb(\SS_1)$, then 
  \begin{enumerate}[label=\rm(\roman*)]
  \item
    $\gp(G_1) = \gp(G_2) = 1$,
  \item
    $\ep(G_1)\ep(G_2) = 0$, and
  \item
    $\eta(G_1, G_2) = 2$ if and only if $xy \in E(G)$.
  \end{enumerate}
\end{lemma}

\begin{proof}
  Suppose that $G \in \Forb(\SS_1)$.
  If $\gp(G_1) \ge 2$, then $G_1^+$ contains a toroidal obstruction.
  Since $G_1^+$ is a proper minor of $G$, this contradicts the fact that $G \in \Forb(\SS_1)$.
  Thus $\gp(G_1) \le 1$ and $\gp(G_2) \le 1$ by symmetry.
  If $\gp(G_1) = 0$,
  then $\gp(G) \le 1$ by Theorem~\ref{th-decker-ori}, a contradiction%
  \footnote{The fact that $\gp(G_1)$ and $\gp(G_2)$ are at least 1 is a simple observation, see for example~\cite{mohar-book}.}.
  We conclude that $\gp(G_1) = 1$ and also $\gp(G_2) = 1$ by symmetry. This shows (i).

  If $\ep(G_1)\ep(G_2) = 1$, then it follows from Theorem~\ref{th-decker-ori} that
  $$g(G) \le \gp(G_1) + \gp(G_2) - \ep(G_1)\ep(G_2) = 1,$$
  a contradiction. Thus $\ep(G_1)\ep(G_2) = 0$ and (ii) holds.

  To show (iii), suppose that $xy \not\in E(G)$ and $\eta(G_1, G_2) = 2$.
  By (i) and (ii), this is only possible if $g(G_1) = g(G_2) = 0$.
  By Theorem~\ref{th-decker-ori}, $g(G) \le g(G_1) + g(G_2) + 1 \le 1$, a contradiction.
  The other implication follows from Lemma~\ref{lm-edge-xy}.
\end{proof}

\begin{figure}
  \centering
  \includegraphics{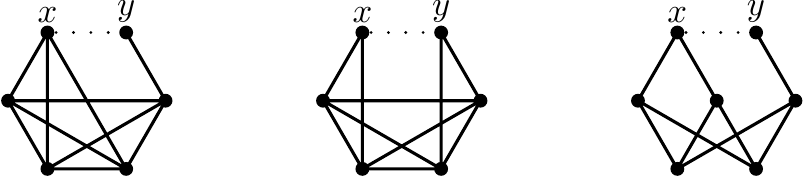}
  \caption{$\T_5$, splits of Kuratowski graphs which belong to $\Cc_0(\gap) \sm \Cc_0(\galt)$}
  \label{fg-split}
\end{figure}
It is time to present the main theorem of this section that derives a full characterization of the obstructions for the torus of connectivity 2.
It can be viewed as an application of Theorem~\ref{th-general} with the outcome summarized in Table~\ref{tb-torus-side}.

\begin{theorem}
  \label{th-torus}
  Suppose that $G$ is an $xy$-sum of connected graphs $G_1$ and $G_2$ and that the following statements hold:
  \begin{enumerate}[label=\rm(\roman*)]
  \item
    $G_1 \in \Cc_0(\gp)$,
  \item 
    $G_2 \in \Cc_0(\galt) \cup \Cc_0(\gap)$,
  \item
    $xy \in E(G)$ if and only if $G_1 \not\in \Cc_0(\g)$ and $G_2 \not\in \Cc_0(\galt)$, and
  \item
    if $\theta(G_1) = \theta(G_2) = 0$, then $G_2 \in \Cc_0(\gap)$.
  \end{enumerate}
  Then $G \in \Forb(\SS_1)$. 
  Furthermore, every obstruction for the torus of connectivity 2 can be obtained this way.
\end{theorem}

\begin{proof}
  The proof consists of two parts. In the first part, we prove that each graph satisfying the conditions (i)--(iv)
  is an obstruction for the torus. In the second part, all obstructions of connectivity 2 are
  shown to be constructed this way.

  Let us assume that (i)--(iv) holds.
  To show that $G$ is an obstruction for the torus, we need to prove that $G_1$, $G_2$, and $xy$ (when $xy \in E(G)$) are minor-tight and that $g(G) = 2$.
  By (i) and Lemma~\ref{lm-alt-kuratowski}, $\ep(G_1) = 1$ and $\gp(G_1) = 1$.
  By (ii) and Corollary~\ref{cr-theta-torus}, $\ep(G_2) = 0$ and $\gp(G_2) = 1$. 
  Hence $\h_1(G) = 2$.
  If $\eta(G_1, G_2) = 2$, then $\theta(G_1) = \theta(G_2) = 1$.
  Thus $G_1 \in \Cc_0(\gp) \sm \Cc_0(\g)$ by Lemma~\ref{lm-alt-kuratowski} and
  $G_2 \in \Cc_0(\gap) \sm \Cc_0(\galt)$ by Corollary~\ref{cr-theta-torus}.
  By (iii), $xy \in E(G)$. 
  Consequently, we have either $\eta(G_1, G_2) \le 1$ or $xy \in E(G)$.
  This excludes the case that $\eta(G_1, G_2) = 2$ and $xy \not\in E(G)$ and we shall use it below.
  If $xy \in E(G)$, then by Theorem~\ref{th-decker-ori}, $g(G) = \h_1(G) = 2$ as required.
  Similarly, if $xy \not\in E(G)$ and $\eta(G_1, G_2) \le 1$, then $\h_1(G) \le \h_0(G)$ by~(\ref{eq-eta}).
  Hence $g(G) = \h_1(G) = 2$ by Theorem~\ref{th-decker-ori}.

\begin{table}
  \centering
  \begin{tabular}{|c | c | c | c|}
    \hline
    $xy \in E(G)$ & $\ep(G_2)$ & $\eta(G_1, G_2)$ & $G_1$ \\
    \hline
    \multirow{2}{*}{yes}  & 0 & \multirow{2}{*}{---} & $\Cc_0(\gp)$ \\
    & 1 & & $\Cc_0(\gap)$ \\
    \hline
    \multirow{4}{*}{no} & \multirow{2}{*}{0} &  0 & $\Cc_0(\gp)$ \\
     & &  1 & $\Cc_0(g)$ or $\Cc_0(\gp)$ \\
    \cline{2-4}
    & \multirow{2}{*}{1} & 0 & $\Cc_0(\gap)$ \\
    & & 1 & $\Cc_0(\galt)$ or $\Cc_0(\gap)$ \\
    \hline
  \end{tabular}
  \caption{Classification of the parts of obstructions for the torus.}
  \label{tb-torus-side}
\end{table}

  It remains to prove minor-tightness.
  Since $\ep(G_2) = 0$ and $G_1 \in \Cc_0(\gp)$, Theorem~\ref{th-general} gives that $G_1$ is minor-tight.
  If $G_2 \in \Cc_0(\gap)$, then $G_2$ is minor-tight by Theorem~\ref{th-general} since $\ep(G_2) = 1$.
  Otherwise, $G_2 \in \Cc_0(\galt) \sm \Cc_0(\gap)$ and $\theta(G_2) = 0$ by Corollary~\ref{cr-theta-torus}.
  Thus $\theta(G_1) = 1$ by (iv) and we have that $\eta(G_1, G_2) = 1$.
  We conclude that $G_2$ is minor-tight by Theorem~\ref{th-general}.

  If $xy \in E(G)$, then (iii) implies that
  $G_1 \in \Cc_0(\gp) \sm \Cc_0(\g)$ and
  $G_2 \in \Cc_0(\gap) \sm \Cc_0(\galt)$.
  Therefore, $\theta(G_1) =1$ by Lemma~\ref{lm-alt-kuratowski} and $\theta(G_2) = 1$ by Corollary~\ref{cr-theta-torus}.
  Hence $\eta(G_1, G_2) = 2$.  Lemma~\ref{lm-alt-kuratowski} applied to $G_1$ implies that $g(G_1 / xy) < \gp(G_1)$.
  Thus $xy$ is minor-tight in $G$ by Lemma~\ref{lm-edge-xy}.
  We conclude that $G$ is an obstruction for the torus by Lemma~\ref{lm-minor-tight}.

  Let us now prove that, for a graph $G \in \Forb(\SS_1)$ of connectivity 2, there exists a 2-vertex-cut $\{x,y\}$ 
  such that when $G$ is viewed as an $xy$-sum of graphs $G_1$ and $G_2$,
  the statements (i)--(iv) hold. We pick $x$ and $y$ as guaranteed by Lemma~\ref{lm-dumbbell-no} so that $G_1, G_2 \not\in \D$.
  Since $G$ is an obstruction, the subgraphs $G_1$, $G_2$, and $xy$ (if present) are minor-tight.
  By Lemma~\ref{lm-torus-prop}, $\gp(G_1) = \gp(G_2) = 1$ and $\ep(G_1)\ep(G_2) = 0$. 
  We may assume by symmetry that $\ep(G_2) = 0$.
  By Corollary~\ref{cr-necessary-obstr}(ii), the graph $G_1$ belongs to $\Cc_0(g) \cup \Cc_0(\gp) = \Cc_0(\gp)$ since $g(G_1) \le \gp(G_1) = 1$.
  Hence (i) holds.
  By Lemma~\ref{lm-alt-kuratowski}, $\ep(G_1) = 1$.

  Since $\ep(G_2) = 0$, Lemma~\ref{lm-alt-kuratowski} gives that $G_2 \not\in \Cc_0(\gp)$.
  By Corollary~\ref{cr-necessary-obstr}(i), the graph $G_2$ belongs to $\Cc_0(\galt) \cup \Cc_0(\gap)$ since $G_2 \not\in \D$, $\gp(G_2) = 1$, and $\gp$ bounds all the other parameters.
  Thus (ii) is true.

  We prove equivalence in (iii) at once.
  By Lemma~\ref{lm-torus-prop}(iii), we have that $xy \in E(G)$ if and only if $\eta(G_1, G_2) = 2$.
  Since $\ep(G_2) = 0$, $\eta(G_1, G_2) = 2$ if and only if $\theta(G_1) = \theta(G_2) = 1$.
  By Lemma~\ref{lm-alt-kuratowski} and Corollary~\ref{cr-theta-torus}, $\theta(G_1) = \theta(G_2) = 1$ if and only if
  $G_1 \in \Cc_0(\gp) \sm \Cc_0(\g)$ and
  $G_2 \in \Cc_0(\gap) \sm \Cc_0(\galt)$.
  We conclude that (iii) holds.

  For (iv), suppose that $G_2 \not\in \Cc_0(\gap)$.
  Since $G_2 \not\in \Cc_0(\gp)$ and $G_2$ is minor-tight, Theorem~\ref{th-general} gives that $\eta(G_1, G_2) = 1$ (as $\H^1_0 \ss \Cc_0(\gp)$ by Lemma~\ref{lm-hopper-1}).
  We conclude that either $\theta(G_1) = 1$ or $\theta(G_2) = 1$ and thus (iv) holds.
  This finishes the proof of the theorem.
\end{proof}

\begin{corollary}
\label{cr-torus}
  There are 68 obstructions for the torus of connectivity 2.
\end{corollary}

\begin{proof}
  By Theorem~\ref{th-torus}, for each $G \in \Forb(\SS_1)$ of connectivity 2, there exists
  a 2-vertex-cut $\{x, y\}$ such that $G$ is an $xy$-sum of parts $G_1$ and $G_2$ satisfying
  (i)--(iv). Let us count the number of graphs in $\Forb(\SS_1)$ of connectivity 2 by 
  counting the number of non-isomorphic $xy$-sums satisfying (i)--(iv).

  Let us first count the number of pairs $G_1$ and $G_2$ for which (i), (ii), and (iv) of Theorem~\ref{th-torus} hold.
  The graphs in $\T_1$ are disconnected so their 2-sum with $G_1$ is not 2-connected.
  The number of connected graphs in $\Cc_0(\galt) \cup \Cc_0(\gap)$ is $|\T_2 \cup \cdots \cup \T_6| = 27$ and
  the number of graphs in $\Cc_0(\gp)$ is 3.
  Thus we have precisely 81 pairs satisfying (i) and (ii). However, some of them do not satisfy (iv).

  Let us consider property (iv).
  There is only a single graph in $\Cc_0(\gp)$ that has $\theta$ equal to 0 (Fig.~\ref{fg-c0-gp}(c)).
  By Lemma~\ref{lm-ti-membership}, there are precisely $|\T_4| = 5$ graphs in $\Cc_0(\galt) \sm \Cc_0(\gap)$;
  they all have $\theta$ equal to 0 by Corollary~\ref{cr-theta-torus}.
  Thus 5 pairs out of the total of 81 do not satisfy (iv) of Theorem~\ref{th-torus} giving the total of 76 pairs satisfying (i), (ii), and (iv).

  For fixed graphs $G_1$ and $G_2$ in $\Gcxy$,
  there are four different $xy$-sums with parts $G_1$ and $G_2$; 
  there are two ways how to identify two graphs on two vertices and the edge $xy$ is either present or not.
  Precisely two of those $xy$-sums satisfy (iii) as the presence of $xy$ depends only on $G_1$ and $G_2$.
  Since for each graph in $\Cc_0(\gp)$ there is an automorphism exchanging the terminals,
  there is precisely one $xy$-sum with parts $G_1$ and $G_2$ that satisfies (i) and (iii).

  Therefore, for each of the 76 pairs, there is a unique $xy$-sum satisfying (i)--(iv).
  By Theorem~\ref{th-torus}, each such $xy$-sum is an obstruction for the torus.
  Some of the obtained obstructions are isomorphic though. 
  Let $G$ be an $xy$-sum of $G_1$ and $G_2$ and $G'$ be an $x'y'$-sum of $G_1'$ and $G_2'$ such that both $G$ and $G'$ satisfy (i)--(iv)
  and there is an isomorphism $\psi$ of $\hat{G}$ and $\hat{G'}$. If $\psi(\{x,y\}) \not= \{x',y'\}$, then $\psi(\{x,y\})$ is a 2-vertex-cut in $G'$.
  It is not hard to see that $G'$ has another 2-vertex cut only if $G_2' \in \T_5$.
  We can see that the preimage of $\psi$ of one side of $\psi(\{x,y\})$ is a graph in $\Cc_0(\gp) \sm \Cc_0(g)$.
  Therefore, $G_1 \in \Cc_0(\gp) \sm \Cc_0(g)$ and $G_2 \in \T_5 \ss \Cc_0(\gap) \sm \Cc_0(\galt)$.
  By (iii), $xy \in E(G)$. But $\psi(x)$ is not adjacent to $\psi(y)$, a contradiction.

  We may assume now that $\psi(\{x, y\}) = \{x', y'\}$. If $\psi(V(G_1)) = V(G_1')$, then $G_1 \iso G_1'$ and $G_2 \iso G_2'$ as argued above.
  Thus $\psi(V(G_1)) \not= V(G_1')$. 
  It is not hard to check that only the graphs in $\T_2$ have a subgraph isomorphic to a graph in $\Cc_0(\gp)$.
  There are $18$ pairs $G_1, G_2$ such that $G_1 \in \Cc_0(\gp)$ and $G_2 \in \T_2$; but there are precisely $10$ non-isomorphic obstructions
  for the torus obtained from these 18 pairs.
  We conclude that there are 68 non-isomorphic obstructions for the torus of connectivity 2.
\end{proof}

\section{Open problems}

\begin{figure}
  \centering
  \includegraphics{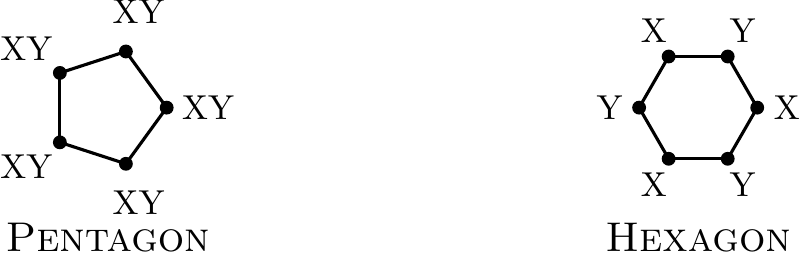}
  \caption{The XY-labelled representation of  $\T_6 \ss \Cc_0(\gap) \sm \Cc_0(\galt)$.}
  \label{fg-alt-xy}
\end{figure}
The following questions remain unanswered:
\begin{enumerate}[label=(\arabic*)]
\item 
  Do hoppers exist? If they do, what is the smallest genus $k$ such that the class $\H^0_k$ ($\H^1_k$, or $\H^2_k$) is non-empty?
\item
  Is it possible that there exists a graph $G \in \Cc(g)$ with $\theta(G) = 1$? In other words, can the graphs $G$ and $G^+$
  be both obstructions for an orientable surface? What is the smallest $k$ such that this is the case for a graph of genus $k$?
\item
  What is the smallest $k(r)$ such that there exists an $r$-connected obstruction $G$ of genus $k$ with a pair of vertices $x, y$
  such that $G$ is {\em not\/} $xy$-alternating. For example, $k(0) = 2$. We do not know the value $k(r)$ for $r > 0$.
\end{enumerate}

\bibliographystyle{habbrv}
\bibliography{bibliography}

\end{document}